\documentclass[12pt,twoside]{amsart}

\usepackage{hyperref}
\usepackage{amsthm, amsmath, amscd, amssymb,centernot,txfonts}
\usepackage[all]{xy}
\usepackage[left=2cm,top=2.5cm,bottom=3cm,right=3cm]{geometry}

\theoremstyle{plain}
\numberwithin{equation}{section}
\newtheorem{theorem}{Theorem}[section]
\newtheorem{proposition}[theorem]{Proposition}
\newtheorem{lemma}[theorem]{Lemma}
\newtheorem{corollary}[theorem]{Corollary}

\newtheorem{remark}[theorem]{Remark}
\newtheorem{example}[theorem]{Example}
\newcommand{\mlkbr}{M_L \otimes K \otimes B^{\otimes r}}
\newcommand{\mlkbrr}{M_L \otimes K \otimes B^{\otimes r+1}}
\newcommand{\mlkb}{M_L \otimes K \otimes B}
\newcommand{\kb}{K \otimes B}
\newcommand{\te}{\otimes}

\begin{document}

\title{Syzygies of surfaces of general type}

\author[P. Banagere]{P. Banagere}
\email{banagere@gmail.com}

\author[Krishna Hanumanthu]{Krishna Hanumanthu}
\address{Department of Mathematics, University of Kansas, Lawrence, KS 66045}
\address{Chennai Mathematical Institute, H1 SIPCOT IT Park, Siruseri, Kelambakkam 603103, India}
\email{krishna@cmi.ac.in}

\subjclass[2010]{13D02, 14C20, 14J29}
\thanks{The second author was partially supported by Robert D. Adams Visiting Assistant Professorship fund at the University of Kansas.}

\maketitle

\begin{abstract}
We prove new results on projective normality, normal presentation and higher syzygies for a surface of general type $X$ embedded by adjoint line bundles $L_r = K + rB$, where  $B$ is a base point free, ample line bundle. Our main results determine the $r$ for which $L_r$ has $N_p$ property. In corollaries, we will relate the bounds on $r$ to the regularity of $B$. Examples in the last section show that several results are optimal. 


\end{abstract}

\section*{Introduction}

The topic of syzygies of algebraic varieties has evoked considerable interest due to, among other things,  the potential interactions it provides between commutative algebra and algebraic geometry. Projective normality and normal presentation of varieties have been studied since the time of Italian geometers. 
Mark Green brought a new perspective by viewing the classical results on projective normality and normal presentation for curves as particular cases of a more general phenomenon: the linearity of syzygies associated to the minimal free resolution of the homogeneous coordinate ring of the variety embedded in a projective space. It is hard in general to write down the full minimal free resolution, but the linearity of the syzygies has attracted attention.
This leads to the notion of $N_p$ property. 

We give a quick introduction to $N_p$ property.  

Let $k$ be an algebraically closed field of characteristic 0. All our varieties are projective, smooth and defined over $k$. 

Let $L$ be a very ample line bundle on a projective variety $X$. Then $L$ determines an embedding of $X$ into the projective space $\mathbb{P}\big{(}H^0(X,L)\big{)}$.  We denote by $S$ the homogeneous coordinate ring of this projective space. Then the {\it section ring} 
$R(L)$ of $L$ is defined as $\bigoplus_{n=0}^{\infty} H^0(X, L^{\otimes n})$ and it is a finitely generated graded $S$-module.  One looks at the minimal graded free resolution of $R(L)$ over $S$:
$$...\rightarrow E_i \rightarrow ... \rightarrow E_2  \rightarrow E_1 \rightarrow E_0 \rightarrow R \rightarrow 0$$
where $E_i  = \bigoplus S(-a_{i,j})$ for all $i \geq 0$ and $a_{i,j}$ are some nonnegative integers.

We say that $L$ has $N_0$ property if $E_0 = S$. This simply means that the embedding determined by $L$ is projectively normal. It is also common to say, in this case, that $L$ is {\it normally generated}. 

$L$ is said to have $N_1$ property if $E_0 = S$ and $a_{1,j} = 2$ for all $j$. In this case, we also say that $L$ is {\it normally presented.}  Geometrically, this means that the embedding is cut out by quadrics. 

For $p \geq 2$, we say that $L$ has $N_p$ property of $E_0 = S$ and $a_{i,j} = i+1$ for all $i = 1, \ldots ,p$. 



Given a very ample line bundle $L$, it is natural to ask for which $p \ge 0$, $L$ has $N_p$ property. This was studied in the case of curves by Green in  \cite{12}. His result shows that if the curve has genus $g$, then $L$ has $N_p$ property if deg($L) \ge 2g+1+p$ (see also \cite{GL}).  This result suggests the possibility of a numerical characterization of $N_p$ property. 

The notion of degree of a line bundle provides a good measure for the positivity of a line bundle on a curve. But, in general, there is no such simple and straight forward measure on surfaces and higher dimensional varieties. A way out is to narrow down the linear systems to adjoint linear systems, that is, linear systems of the form $K\te B^{\te r}$, where $K=K_X$ is a canonical divisor of $X$ and $B$ is an ample line bundle on $X$. \textit{Mukai's conjecture} deals with precisely such line bundles and asks for the linearity of syzygies associated to the embeddings given by $K\te B^{\te r}$. 

On a surface, given an ample line bundle $B$, Reider \cite{13} proved that 
$L = K_X \otimes B^{\otimes r}$ is very ample for $r \ge 4$. Then it is natural to ask if it is projectively normal.  
More generally, Mukai conjectured that on a surface, if $B$ is an ample line bundle, then 
$K_X \otimes B^{\otimes r}$ has $N_p$ property for $r \ge  p+4$. 
If $B$ is very ample, the situation is tractable, and there is a nice general result of Ein and Lazarsfeld \cite{5} for projective varieties of all dimensions. If we assume $B$ to be only ample and base point free, many difficulties that one sees in the ample case still persist and it is still a challenge to prove results along the lines of Mukai's conjecture. 

To a large extent, the situation for surfaces of Kodaira dimension less than 2  is clarified. Projective normality for rational surfaces was studied in \cite{H}. \cite{9} looks at higher syzygies on rational surfaces and a strong conjecture, which implies Mukai's conjecture, is made and proved for rational surfaces. Syzygies of ruled surfaces are investigated in \cite{1, 6,7,Pa1,Pa2}. Projective normality on Enriques surfaces is studied in \cite{GLM}. In \cite{4}, effective bounds toward Mukai's conjecture are obtained for surfaces of Kodaira dimension zero.  
See \cite{4,8} for a summary and other results.

Many technical difficulties, unseen for surfaces of smaller Kodaira dimension, arise for the case of surfaces of general type because $K$ becomes more positive. 
In \cite{2}, Mukai's conjecture is proved for large families of line bundles on surfaces of general type. More precisely, $N_p$ property was studied for line bundles of the form $K \te B^{\te r}$ with $B$ ample and base point free under the hypothesis that $B^2 \ge B \cdot K$ (or $B-K$ is nef, for $p \ge 2$).

In this article, we deal with general results about higher syzygies for surfaces of general type and ample and base point free line bundles $B$ for which $B^2 < B\cdot K$.

We prove projective normality and normal presentation for $K + rB$  under the hypothesis $B^2 \ge \frac{a}{b}(B \cdot K)$ for two integers $a < b$. This is a generalization of the assumption $B^2 \ge B \cdot K$ made in \cite{2} to establish $N_0$ and $N_1$ properties. Under the hypothesis that $nB-K$ is nef for an integer $n \ge 2$, we prove $N_p$ property for $K+rB$ where $r$ is a linear function of slope 2 in $n$ and $p$. Finer results are obtained under the  hypothesis that $X$ is regular. 



The line bundles $B$ satisfying $B-K$ nef form large families in the Pic, but this still leaves out infinitely many families for which $B-K$ is not nef. 
A natural question is to look at those families of line bundles $B$ for which $(n+1)B- K$ is nef but $nB-K$ is not nef for some integer $n \ge 1$. It must be noted that there are large classes of examples of surfaces of general type with huge moduli on which there are infinitely many line bundles for which these technical conditions are satisfied. We describe such examples in the last section of this paper. A Mukai type result for such line bundles not only generalizes the earlier results but also yields an asymptotic version of Mukai's conjecture for all surfaces of general type. Our results also provide a context for earlier results and point towards further progress for a complete solution to the conjecture of Mukai.

Castelnuovo-Mumford regularity is a notion that interests both algebraic geometers and commutative algebraists.  For a base point free ample line bundle $B$ on a projective variety $X$, the integers $r$ such that $B^{\te r}$ has $N_p$ property are determined in terms of the Castelnuovo-Mumford regularity of $B$ in \cite[Theorem 1.3]{4}. 
We prove analogous results for adjoint line bundles on 
surfaces of general type $X$. The results do not follow from just general arguments using 
Castelnuovo-Mumford regularity because of the positivity of $K$ (especially for surfaces of general type that are regular) and require different methods.  

The methods that are developed to handle the problems of the kind mentioned above include vector bundle techniques on curves and surfaces; semi-stability of vector bundles on curves and surfaces, multiplication maps of global sections and proving some positivity statements for vector bundles. Other techniques involve proving some inequalities for intersections numbers on surfaces, homological algebra and use of the notion of Castelnuovo-Mumford regularity. The general results that are true for all surfaces of general type use some subtle inductive statements in homological algebra and the so-called Castelnuovo-Mumford lemma. But for regular surfaces of general type, a delicate analysis of linear systems on surfaces gives finer results. 
Here one can reduce the problem of multiplication maps on surfaces to multiplication maps of certain special curves on the surface, the choice of which is not always canonical. Then vector bundle techniques involving stability arguments, and the general methods mentioned above give finer results. 

The structure of the paper is as follows. 

In Section \ref{prelims}, we list some results that are needed later on. Many of the results are classical and stated without proof. 

In Section \ref{N0}, we establish projective normality and normal presentation under the hypothesis that $B^2 \ge \frac{a}{b}(B \cdot K)$ for $a < b$.  
In subsection \ref{reg1}, we obtain better results on regular surfaces by restricting to curves.

In Section \ref{higher}, we study $N_p$ property for $p \ge 2$ under the hypothesis that $nB-K$ is nef or  $nB-2K$ is nef for some $n \geq 2$.  
In \ref{reg2}, we prove stronger results for regular surfaces. 

Finally in Section \ref{examples}, we give 	examples of infinitely many families of surfaces of general type when our hypotheses are satisfied. Some of these examples show that our arguments are optimal. Given $n \ge 2$, we construct examples of infinite families of surfaces and line bundles $B$ on them such that $(n+1)B-K$ is nef and $nB-K$ is \textit{not} nef, thereby proving that our results do not follow from \cite{2}.   


\noindent
{\bf Notation:} Unless otherwise stated, $X$ represents a smooth minimal surface of general type over an algebraically closed field $k$ of characteristic 0. $K$ or $K_X$ denotes its canonical divisor. We will use the multiplicative notation of line bundles and additive notation of divisors interchangeably. Thus when $L$ is a line bundle, $L^{\otimes r}$ and $rL$ are the same.

We write $H^i(L)$ for the cohomology group $H^i(X,L)$. The dimension of $H^i(L)$ as  a vector space over $k$ is denoted by 
$h^i(L)$. $A \cdot B$ refers to the intersection number of two divisors $A$ and $B$. A line bundle $L$ is called \textit{nef} if $L \cdot C \ge 0$ for every irreducible curve $C \subset X$. 
$L$ is called \textit{big} if for some $m > 0$, the mapping defined by $L^{\te m}$ is birational onto its image in the projective space.  

\subsection*{Acknowledgment} We thank the referee for a careful reading and several comments that improved the exposition.

\section{Preliminaries}\label{prelims}
In this section, we state several results that will be used repeatedly in the sequel.  We only prove Lemmas \ref{1.1} and \ref{horace}. The others are from various sources, but we state them here for completeness. 

Given a vector bundle $F$ on a projective variety $X$ that is generated by its global sections, we have the canonical surjective map:\\
\begin{eqnarray}
H^0(F) \otimes {\mathcal O}_{X} \rightarrow F.
\end{eqnarray}

Let $M_F$ be the kernel of this map. We have then the natural exact sequence:
\begin{eqnarray}\label{cansurj}
0 \rightarrow M_F \rightarrow H^0(F) \otimes {\mathcal O}_{X} \rightarrow F \rightarrow 0.
\end{eqnarray}

We are going to study $N_p$ property for adjoint line bundles on a surface of general type. The following characterization of $N_p$ property will be used. 

\begin{theorem}\label{np}\textup{\cite[Lemma 1.6]{5}}
Let $L$ be a very ample line bundle on a projective variety $X$. Assume that $H^1(L^{\otimes k}) = 0$ for all $k \ge 1$. Then $L$  satisfies $N_p$ property 
if and only if $H^1(M^{\otimes a}_L \otimes L^{\otimes b}) = 0$ for all $1 \le a \le p+1$ and $b \ge 1$. 
\end{theorem}

The proof of Lemma 1.6, \cite{5} in fact shows that this result holds for ample and base point free line bundles $L$. We will use this version in this paper. 

The following useful remarks will be used repeatedly. Let $F$ be a globally generated vector bundle and let $A$ be any line bundle on a projective variety $X$. 
\begin{remark}\label{00}{\rm 
$H^1(M_F \otimes A) = 0$ if the following two conditions hold.
\begin{itemize}
\item The multiplication map $H^0(F) \otimes H^0(A) \rightarrow H^0(F \otimes A)$ is surjective.
\item $H^1(A) = 0$. 
\end{itemize}

This is easy to see: tensor the sequence (\ref{cansurj}) by $A$ and take global sections:

$..\rightarrow H^0(F) \otimes H^0(A) \rightarrow H^0(F \otimes A) \rightarrow H^1(M_F \otimes A)
\rightarrow H^0(F) \otimes H^1(A) \rightarrow ...$}
\end{remark}
\begin{remark}\label{000}{\rm 
$H^2(M_F \otimes A) = 0$ if the following two conditions hold.
\begin{itemize}
\item $H^1(F \otimes A) = 0$.
\item $H^2(A) = 0$.
\end{itemize}
This is easy to see: tensor the sequence (\ref{cansurj}) by $A$ and take global sections:

$..\rightarrow H^1(F \otimes A) \rightarrow H^2(M_F \otimes A)
\rightarrow H^0(F) \otimes H^2(A) \rightarrow ...$}
\end{remark}

\begin{lemma}\label{1.1}
Let $X$ be a surface of general type. Let $B$ be an ample base point free divisor on $X$ with $B^2 \geq \frac{a}{b}(B \cdot K)$ for positive integers $a < b$. Suppose that $H^1(mB) = 0$ for some $m > \frac{b}{a}$. Then $H^1(lB) = 0$ for all $l \geq m$.
\end{lemma}
\begin{proof} Let $C \in |B|$ be a smooth irreducible curve. 
Let $D = lB|_C$ for some $l \geq m+1$. Then we have
$deg(D) = lB^2 = (l-1)B^2 + B^2 \geq \frac{ma}{b} B \cdot K + B^2 > B\cdot K + B^2 = 2g(C)-2.$
So $deg(K_C - D) < 0$ and $H^1(D) = H^0(K_C - D) = 0$.
Now consider the short exact sequence:  
$0 \rightarrow \mathcal{O}(-B) \rightarrow \mathcal{O} \rightarrow {\mathcal{O}}_C \rightarrow 0$. 

Tensoring with $(m+1)B$ and taking global sections, we obtain:

$..\rightarrow H^1(mB) \rightarrow H^1((m+1)B)  \rightarrow H^1((m+1)B|_C) \rightarrow ..$

$H^1(mB) = 0$, by assumption and we have shown above that $H^1((m+1)B|_C) = 0$. So 
$H^1((m+1)B) = 0$. Now it follows easily, by induction, that $H^1(lB) = 0$ for all $l \geq m+1$.
\end{proof}
\begin{corollary}
Let $X$ be a surface of general type. Let $B$ be an ample base point free divisor on $X$ with $B^2 \geq \frac{1}{n}(B \cdot K)$ for a positive integer $n \ge 2$. Suppose that $H^1(mB) = 0$ for some $m \geq n+1$. Then $H^1(lB) = 0$ for all $l \geq m$.
\end{corollary}
\begin{proof} Follows immediately from the lemma by setting $a = 1, b =n$.\end{proof}

\begin{lemma}\label{1.2}
Let $X$ be a surface with nonnegative Kodaira dimension and let $B$ be an ample line bundle. If $B^2 \ge \frac{a}{b} (B \cdot K)$ for positive integers $a < b$, then 
$B\cdot K \ge \frac{a}{b} (K^2)$.
\end{lemma}
\begin{proof}
This follows easily from \cite[Lemma 2.2]{2}. Indeed, let $B' = bB$ and $m = a$. Then apply \cite[Lemma 2.2]{2} to $B', m$ to obtain the lemma. 
\end{proof}

\begin{lemma}\label{horace}
Let $E$ and $L_1, L_2,...,L_r$ be coherent sheaves on a variety $X$. Consider the multiplication maps

$\psi: H^0(E) \otimes H^0(L_1 \otimes ...\otimes L_r) \rightarrow H^0(E \otimes L_1 \otimes ... \otimes L_r)$,

$\alpha_1: H^0(E) \otimes H^0(L_1) \rightarrow H^0(E \otimes L_1)$,

$\alpha_2: H^0(E \otimes L_1) \otimes H^0(L_2) \rightarrow H^0(E \otimes L_1 \otimes L_2)$,

...,

$\alpha_r: H^0(E\otimes L_1 \otimes ... \otimes L_{r-1}) \otimes H^0(L_r) \rightarrow H^0(E \otimes L_1 \otimes ... \otimes L_{r})$.

If $\alpha_1$,...,$\alpha_{r}$ are surjective, then so is $\psi$.
\end{lemma}
\begin{proof}
We have the following commutative diagram where $id$ denotes the identity morphism:
\begin{displaymath}
\xymatrix @R=2pc @C=3pc{
H^0(E) \otimes H^0(L_1) \otimes ... \otimes H^0(L_r) \ar[r]^{\alpha_1 \otimes id} \ar[d]^{\phi} & H^0(E \otimes L_1) \otimes H^0(L_2) \otimes ... \otimes H^0(L_r) \ar[d]^{\alpha_2 \otimes id} \\
H^0(E) \otimes H^0(L_1 \otimes ... \otimes L_r)    \ar[dd]^{\psi} & H^0(E \otimes L_1 \otimes L_2) \otimes H^0(L_3) \otimes ... \otimes H^0(L_r)  \ar[d]^{\alpha_3 \otimes id} \\
   &...\ar[d]^{\alpha_{r-1} \otimes id}\\
H^0(E \otimes L_1 \otimes ... \otimes L_r)   &\ar[l]_{\alpha_r} H^0(E \otimes L_1 \otimes ... \otimes L_{r-1}) \otimes H^0(L_r)}
\end{displaymath}

Since $\alpha_1, \alpha_2,...,\alpha_r$ are surjective and this diagram is commutative, a simple diagram chase shows that  $\psi$ is surjective.
\end{proof}

We will now state some results without proof which will be used often in this paper. 
The first one is the Castelnuovo - Mumford lemma.  We remark that though Mumford stated this fact under the hypothesis that $E$ is ample and base point free, the proof works with only the base point free assumption. 

\begin{lemma}\label{key} \textup{\cite[Theorem 2]{3}}
Let $E$ be a base-point free line bundle on a projective variety $X$ and let $F$ be a coherent sheaf on $X$. If $H^i(F \otimes E^{-i}) = 0$ for $i \geq 1$, then the multiplication map 
$$H^0(F \otimes E^{\otimes i}) \otimes H^0(E) \rightarrow H^0(F \otimes E^{\otimes i+1})$$
is surjective for all $i \geq 0$. 
\end{lemma}		

We refer to this result as CM lemma in this paper.

\begin{lemma}\label{kv}
{\bf \text (Kawamata - Viehweg vanishing)} Let $X$ be a nonsingular projective variety over the complex number field $\mathbb{C}$. Let $D$ be a nef and big divisor on $X$. Then $$H^i(K_X \otimes D) = 0, \text {~for~} i > 0.$$ 
\end{lemma}

For a proof, see \cite{10} or \cite{11}. We will refer to this result simply as K-V vanishing.

The next three results will be used to prove the surjectivity of multiplication maps on regular surfaces. 

\begin{lemma}\label{redcurves}
Let $X$ be a regular surface. Let $E$ be a vector bundle and let $C$ be a divisor such that $L = \mathcal{O}_X(C)$ is a base point free divisor and  $H^1(E \otimes L^{-1})= 0$. If the multiplication map $H^0(E \otimes \mathcal{O}_C) \otimes H^0(L \otimes \mathcal{O}_C) \rightarrow H^0(E \otimes L \otimes \mathcal{O}_C)$
is surjective, then the map $H^0(E) \otimes H^0(L) \rightarrow H^0(E \otimes L)$ is also surjective. 
\end{lemma}
\begin{proof}
This is not difficult: e.g., see \cite[Observation 2.3]{4}.
\end{proof} 

Let $E$ be a vector bundle on a curve. Then $\mu(E)$ denotes the \textit{slope} of $E$: 
$\mu(E) =  \frac{degree(E)}{rank(E)}$.
\begin{proposition}\label{butler}\textup{\cite[Proposition 2.2]{1}}
Let $E$ and $F$ be semistable vector bundles on a curve $C$ of genus $g$ such that $E$ is generated by its global sections. Then the multiplication map 
$H^0(F) \otimes H^0(E) \rightarrow H^0(F \otimes E)$ is surjective if the following two conditions hold:
\begin{enumerate}
\item $\mu(F) > 2g$, and
\item $\mu(F) > 2g+rank(E)[2g-\mu(E)]-2h^1(E)$.
\end{enumerate}
\end{proposition}
\begin{lemma}\label{gp}
Let $X$ be a projective variety. Let $L$ be a base point free line bundle and let $Q$  be an effective divisor. Let $q$ be a reduced and irreducible member of $|Q|$. 
Let $R$ be a line bundle and $G$ a sheaf on $X$. Assume that the following two conditions hold:
\begin{enumerate}
\item $H^1(L \otimes Q^{-1}) = 0$, 
\item $H^0(M^{\otimes n}_{L \otimes \mathcal{O}_q} \otimes R \otimes \mathcal{O}_q) \otimes H^0(G) \rightarrow H^0(M^{\otimes n}_{L \otimes \mathcal{O}_q} \otimes R \ \otimes \mathcal{O}_q \otimes G)$ is surjective for some integer $n \ge 1$. \end{enumerate}
Then the following map  is surjective:
 $H^0(M^{\otimes n}_{L} \otimes R \otimes \mathcal{O}_q) \otimes H^0(G) \rightarrow H^0(M^{\otimes n}_{L} \otimes R \ \otimes \mathcal{O}_q \otimes G)$. 
\end{lemma}
\begin{proof}
This follows easily from \cite[Lemma 2.9]{4}.
\end{proof}
\section{Normal Generation and Normal Presentation}\label{N0}
In this section, unless otherwise stated, $X$ is a minimal nonsingular surface of general type with a canonical divisor $K= K_X$ and $B$ will be a base point free, ample line bundle on $X$ such that $K+B$ is base point free. 
\begin{remark} \label{B^2}
\rm{
If $B^2 \ge 5$, then it follows from \cite{13} that $K_X+B$ is ample and base point free 
(see e.g. \cite[Lemma 3.6]{2}). 

Moreover, in general, we have that $B^2 \ge 2$. Indeed, let $C \in |B|$ be a smooth curve. If $B^2 =1$, then $L = B|_C$ is base point free divisor on the curve $C$ and degree of $L$ is 1. This forces $C$ to be rational. 
But $B^2 + K_X \cdot B = 2g(C)-2=-2$. This is absurd, because $B$ is base point free, ample and $K_X$ is nef. 

Thus our hypothesis that $K_X+B$ is base point free is an additional condition only when $B^2 = 2, 3,$ or 4.}
\end{remark}

Our goal is to study the line bundle $K+rB$ for $r \ge 3$ and determine the values of $r$, for which it has $N_0$ and $N_1$ properties. As noted earlier, this typically involves checking the surjectivity of some multiplication maps of global sections of divisors. The general procedure we follow uses Lemma \ref{horace} and CM lemma. In the case when $X$ is regular, however, there are more 
tools available for proving surjectivity. In section \ref{reg1}, we treat this case and obtain better bounds on $r$. 

\begin{remark}\label{n0n1} {\rm 
Let $L = K+rB$ with $r \ge 3$. Then $H^1(L)=0$, by K-V vanishing.
By Theorem \ref{np}, $L$ has $N_p$ property if and only if 
$H^1(M^{\otimes a}_L \otimes L^{\otimes b}) = 0$ for all $1 \le a \le p+1$ and $b \ge 1$. 

Consider the sequence \ref{cansurj} with $F=L$. Tensoring with 
$M^{\otimes (a-1)}_L \te L^{\te b}$ and taking the long exact sequence of cohomology, we see that the above cohomology vanishing is equivalent to the surjectivity of the following map for 
$1 \le a \le p+1$ and $b \ge 1$: 
$$H^0(M^{\otimes (a-1)}_L \te L^{\te b}) \te H^0(L) \rightarrow H^0(M^{\otimes (a-1)}_L \te L^{\te b+1}).$$
Throughout this article, we prove these maps are surjective when $b=1$ and exactly same methods give us the result for $b>1$. We will illustrate this with the case $a=1$. In this case, we have to show the following map is surjective. 
$$H^0(L^{\te b}) \te H^0(L) \rightarrow H^0(L^{\te b+1}).$$


For this, our strategy will be as follows. Using Lemma \ref{horace}, we split $L=K+rB$ into $(r-1)$ copies of $B$ and a copy of $K+B$. So the above map is surjective if the following maps are surjective: 
$$H^0(L^{\te b}) \otimes H^0(B) \rightarrow H^0(L^{\te b} \te B);\hspace{.15in}
H^0(L^{\te b} \te B) \otimes H^0(B) \rightarrow H^0(L^{\te b} \te B^{\te 2}),\ldots, \text{~and}$$
$$H^0(L^{\te b} \te B^{\te r-1}) \te H^0(K+B) \rightarrow H^0(L^{\te b+1}).$$
Then we use either CM Lemma \ref{key} or reduction to curves (Lemma \ref{redcurves}) in order to prove these maps are surjective. In every case we consider, it turns out that it is sufficient to deal with the case $b=1$. If the hypothesis required for either Lemma \ref{key} or Lemma \ref{redcurves} are satisfied when $b=1$, they are also satisfied when $b>1$. 
}
\end{remark}

We first prove a general result on surjectivity of multiplication maps: 

\begin{theorem}\label{0}
Let $X$ be a minimal smooth surface of general type and 
let $B$ be an ample base point free line bundle on $X$ such that $B^2 \geq \frac{a}{b} (B \cdot K)$ for some positive integers $a < b$.
Let $m,n$ be positive integers such that $n \geq 3$.
The multiplication map $$H^0(K + nB) \otimes H^0(K+mB) \rightarrow H^0(2K + (n+m)B)$$ is surjective if the following conditions hold:
\begin{eqnarray}
H^1\big{(}(n+m-2)B\big{)} = 0 ,\label{c1}\\
\frac{(n+m-3)a}{b} > 2 \textbf{~or~} 
{\Huge{\{}}\frac{(n+m-3)a}{b} = 2 \textbf{~and~} 2K \neq (n+m-3)B {\Huge{\}}} \label{c2}
\end{eqnarray}
\end{theorem}
\begin{proof}
We use Lemma \ref{horace} and the first step is to prove the surjectivity of 
$$H^0(K + nB) \otimes H^0(B) \rightarrow H^0(K + (n+1)B).$$
By CM lemma, this follows if $H^1(K+(n-1)B) = 0$ and $H^2(K+(n-2)B) = 0$. Both these hold by K-V vanishing because $n \geq 3$.

We can similarly gather $m-1$ copies of $B$ and only need to prove the surjectivity of the following map:
$$H^0(K + (n+m-1)B) \otimes H^0(K+B) \rightarrow H^0(2K + (n+m)B).$$
Again we use CM lemma. First we show $H^1\big{(}(n+m-2)B\big{)} = 0$. This is the condition (\ref{c1}). Next, we prove $H^2\big{(}(n+m-3)B-K\big{)} = 0$.

By Serre duality, $H^2\big{(}(n+m-3)B-K\big{)} = H^0\big{(}2K- (n+m-3)B\big{)}$. If this last group is nonzero, then there exists an effective divisor $D$ which is linearly equivalent to 
$2K- (n+m-3)B$. Since $B$ is ample, we have $B \cdot D = B \cdot (2K- (n+m-3)B) \geq 0$. 

So $2 B \cdot K \geq (n+m-3) B^2$. By hypothesis, $B^2 \geq \frac{a}{b}(B\cdot K)$. So we get 
$$2 B \cdot K \geq (n+m-3) B^2 \geq \frac{(n+m-3)a}{b}(B\cdot K).$$

This leads to a contradiction if $\frac{(n+m-3)a}{b} > 2$. Assume now that
$\frac{(n+m-3)a}{b} = 2$ and $2K \neq (n+m-3)B$.

Then all the terms in the above inequalities are equal and we have $2 B \cdot K =  (n+m-3) B^2$, which in turn implies that 
$$B \cdot \big{(}2K - (n+m-3)B\big{)} = 0.$$ Since $B$ is ample and 
$2K - (n+m-3)B$ is effective, we obtain that
$2K = (n+m-3)B$. This is a contradiction. 
\end{proof}

Theorem \ref{0} allows us to determine when the line bundle $K+rB$ has $N_0$ property. 

\begin{corollary}\label{n0_1}
Let $X$ be a minimal smooth surface of general type and 
let $B$ be an ample base point free line bundle on $X$ such that $B^2 \geq \frac{a}{b} (B \cdot K)$ for some positive integers $a < b$. Let $L = K + rB$ with $r \geq 3$. Then $L$ satisfies the $N_0$ property if 
\begin{eqnarray}
H^1\big{(}(2r-2)B\big{)} = 0, \label{c3}\\
r > \frac{b}{a}+\frac{3}{2} \textbf{~or~} 
{\Huge{\{}}r = \frac{b}{a}+\frac{3}{2} \textbf{~and~} 2K \neq (2r-3)B {\Huge{\}}} \label{c4}
\end{eqnarray}
\end{corollary}
\begin{proof}
This follows easily from the theorem. By Remark \ref{n0n1}, $L$ satisfies the $N_0$ property if the following map is surjective:
$$H^0((K + rB)) \otimes H^0(K+rB) \rightarrow H^0(2K + 2rB).$$ Applying the theorem with $n = m = r$ we obtain the result.
\end{proof}

We will state another corollary in the special case that $B^2 \ge  \frac{1}{n} (B \cdot K)$ for some $n \geq 2$. 

\begin{corollary}\label{n0_2}
Let $X$ be a minimal smooth surface of general type and 
and
let $B$ be an ample base point free line bundle on $X$ such that $B^2 \geq \frac{1}{n} (B \cdot K)$ for some $n \geq 2$. Let $L = K +rB$. Suppose that $H^1\big{(}(2r-2)B\big{)} = 0$. Then $L$ satisfies the $N_0$ property for $r \geq n+2$.
\end{corollary}
\begin{proof}
Applying Corollary \ref{n0_1} with $a = 1, b = n$, we have that $L$ satisfies $N_0$ property if the condition (\ref{c4}) holds. 

$r \geq n+2 \Rightarrow 2r-3 \geq 2n+1 \Rightarrow \frac{2r-3}{n} \geq 2 + \frac{1}{n} > 2 $. Hence the corollary follows. 
\end{proof}

\begin{remark}{\rm 
As remarked in the introduction, our results relate the bounds on $r$ so that $L_r = K+rB$ 
has $N_p$ property to the Castelnuovo-Mumford regularity of $B$. 

 Let $m \geq 0$. We say that $B$ is \textit{$m$-regular} (with respect to $B$) if 
$H^i(B^{\otimes m+1-i}) = 0, \text{~for~} i>0.$
If $B$ is $m$-regular, then it is $(m+1)$-regular \cite[Theorem 1.8.5.(iii)]{L}. We define the {\it regularity} of $B$ to be $m$ if $B$ is $m$-regular, but {\it not} $(m-1)$-regular.   See \cite[Section 1.8]{L} for more details on this concept.

In this case, by Corollary \ref{n0_1}, it follows that 
if reg$(B) \ge 2r-2$ and $r > \frac{b}{a}+\frac{3}{2}$, then $L$ is normally generated.
}
\end{remark}

Next, we will address the question of normal presentation for $K+rB$. The following theorem gives the conditions required for this. 

\begin{theorem}\label{n1_0}
Let $X$ be a minimal smooth surface of general type.
Let $B$ be an ample base point free line bundle on $X$ such that $B^2 \geq \frac{a}{b} (B \cdot K)$ for some positive integers $a< b$. 
Let $L = K + rB$ for some $r \geq 3$. Then $L$ satisfies the $N_1$ property if the following conditions hold:
\begin{eqnarray}
H^1((2r-3)B) = 0,  \label{c5} \\
r > \frac{b}{a}+2 \textbf{~or~} 
{\Huge{\{}}r = \frac{b}{a}+2 \textbf{~and~} 2K \neq (2r-4)B {\Huge{\}}} \label{c6}
\end{eqnarray}
\end{theorem}
\begin{proof}
Note that \eqref{c6} implies the condition (\ref{c4}). 
Also, since $2r-3 > \frac{b}{a}$, 
Lemma \ref{1.1} and \eqref{c5} imply condition (\ref{c3}). 
Hence by Corollary \ref{n0_1}, $L$ satisfies the $N_0$ property. 

Next we establish $N_1$ property by proving that the following map is surjective (see Remark \ref{n0n1}):

\begin{eqnarray}
H^0(M_L \otimes K \otimes B^{\otimes r}) \otimes H^0(K \otimes B^{\otimes r} ) \rightarrow H^0(M_L \otimes K^{\otimes 2}  \otimes B^{\otimes 2r}). \label{c9}
\end{eqnarray}

As before we use Lemma \ref{horace} and CM lemma. First, we prove the surjectivity of:
$$H^0(\mlkbr) \otimes H^0(B) \rightarrow H^0(\mlkbrr).$$

By CM lemma, this is true if the following two vanishings hold:
\begin{eqnarray}
H^1(\mlkb^{\otimes r-1}) = 0,  \label{c7} \\
H^2(\mlkb^{\otimes r-2}) = 0. \label{c8}
\end{eqnarray}

Using Remark \ref{00}, (\ref{c7}) holds if the following map is surjective 
(since $H^1(K \te B^{\te r-1}) = 0)$:
$$H^0(\kb^{\te r}) \otimes H^0(\kb^{\te r-1}) \rightarrow H^0(K^{\te 2} \te B^{\te 2r-1}).$$

We use Theorem \ref{0} (setting $n = r, m = r-1$). We verify that  (\ref{c1}) and (\ref{c2}) hold. Using Lemma \ref{1.1}, \eqref{c1} follows easily from (\ref{c5}).  
Also, (\ref{c6}) gives $\frac{(2r-4)a}{b} > 2$ or 
$\frac{(2r-4)a}{b} = 2$ and $2K \ne (2r-4)B$. This gives \eqref{c2}.

Using Remark \ref{000}, (\ref{c8}) follows because $H^1(L \te K \te B^{\te r-2} ) = 0$ and $H^2(K \te B^{\te r-2}) = 0$ (by K-V vanishing). 

Next, we prove the surjectivity of:
$$H^0(\mlkb^{\te r+1}) \otimes H^0(B) \rightarrow H^0(\mlkb^{\te r+2}).$$

By CM lemma, this is true if the following two vanishings hold:
\begin{eqnarray*}
H^1(\mlkbr) = 0, \\
H^2(\mlkb^{\te r-1}) = 0.
\end{eqnarray*}

These two follow in exactly the same way as (\ref{c7}) and (\ref{c8}) using the hypotheses (\ref{c5}) and (\ref{c6}). 

Repeating this process, we can prove the surjectivity of the following map:
$$H^0(\mlkb^{\te 2r-2}) \otimes H^0(B) \rightarrow H^0(\mlkb^{\te 2r-1}).$$

To show that (\ref{c9}) surjects, it remains only to show that the following map surjects:
$$H^0(\mlkb^{\te 2r-1}) \otimes H^0(\kb) \rightarrow H^0(M_L \otimes K^{\otimes 2}  \otimes B^{\otimes 2r}).$$

We use CM lemma:
\begin{eqnarray}
H^1(M_L \te B^{\te 2r-2}) = 0, \label{c10}\\
H^2(M_L \te B^{\te 2r-3} \te K^{-1}) = 0\label{c11}.
\end{eqnarray}

(\ref{c10}) follows easily by CM lemma. Because $H^1(B^{\te 3r-3}) = 0$, (\ref{c11}) holds if 
$H^2(B^{\te 2r-3} \te K^{-1} ) = 0$.

By Serre duality, $H^2(B^{\te 2r-3} \te K^{-1} )= H^0(K^{\te 2} \te B^{\te 3-2r})$. 

If this group is nonzero, then 
$2K - (2r-3)B$ is linearly equivalent to an effective divisor. Since $B$ is ample, 
$B \cdot (2K - (2r-3)B) \geq 0$. This implies that $2B\cdot K \geq (2r-3)B^2$. But this contradicts the hypothesis (\ref{c6}), because $\frac{(2r-3)a}{b} > \frac{(2r-4)a}{b} > 2$.

Hence the map (\ref{c9}) is surjective and the theorem is proved.
\end{proof}
\begin{remark}{\rm 
If reg$(B) \ge 2r-3$ and $r > \frac{b}{a}+{2}$, then $L_r$ is normally generated.
}
\end{remark}

\begin{corollary}\label{n1_1}
Let $X$ be a minimal smooth surface of general type and 
let $B$ be an ample base point free line bundle on $X$ such that $B^2 \geq \frac{1}{n} (B \cdot K)$ for some $n \geq 2$. Let $L = K +rB$. Suppose that $H^1((2r-3)B) = 0$.
Then $L$ satisfies $N_1$ property if any one of the following conditions holds:
\begin{eqnarray}
r \geq n+3,\\ 
r = n+2 \text{~and~} 2K \neq (2r-4)B
\end{eqnarray}
\end{corollary}
\begin{proof}
We apply Theorem \ref{n1_0} with $a = 1, b= n$. (\ref{c5}) holds by hypothesis. 

$r \geq n+3 \Rightarrow 2r-4 \geq 2n+2 \Rightarrow \frac{2r-4}{n} \geq 2 +  \frac{1}{n} > 2.$

If $r = n+2$, then $\frac{2r-4}{n} = 2$, but  $2K \neq (2r-4)B$. So the conditions (\ref{c5}) and
(\ref{c6}) hold and $L$ satisfies the $N_1$ property.
\end{proof}

%
%
%
%
%

\subsection{Regular Surfaces}\label{reg1}
In this section, we assume that the surface $X$ is regular. That is: $H^1({\mathcal O}_X) = 0$. As before $B$ is a base point free ample line bundle such that $B^2 \geq \frac{a}{b} (B \cdot K)$ for some positive integers $a < b$. 

We have more tools available to prove surjectivity of multiplication maps on regular surfaces. Using these we obtain the following general result: 
\begin{theorem}\label{1}
Let $X$ be a minimal smooth regular surface of general type and let
$B$ be an ample base point free line bundle on $X$ such that $B^2 \geq \frac{a}{b} (B \cdot K)$ for some positive integers $a < b$. 
Suppose that $p_g = h^0(K) \ge 3$ and $K^2 \geq 2$. 
Let $m,n$ be positive integers such that $n \geq 3$.
The multiplication map $$H^0(K + nB) \otimes H^0(K+mB) \rightarrow H^0(2K + (n+m)B)$$ is surjective if the following conditions hold:
\begin{eqnarray}
\frac{(n+m-2)a^2}{b^2} + \frac{(n+m-4)a}{b} \geq 2, \label{22}\\
H^1\big{(}(n+m-2)B\big{)} = 0 \label{22a}
\end{eqnarray}
\end{theorem}
\begin{proof}
Initially, we will use Lemma \ref{horace}, CM lemma and K-V vanishing as in the proof of Theorem \ref{0}.

First step is to prove the surjectivity of 
$$H^0(K + nB) \otimes H^0(B) \rightarrow H^0(K + (n+1)B).$$
This follows because $H^1(K+(n-1)B = 0$ and $H^1(K+(n-2)B = 0$ by K-V vanishing (since $n \geq 3$).

After repeating this $m-1$ times, we are left with the following map:
\begin{eqnarray}
H^0(K + (n+m-1)B) \otimes H^0(K+B) \rightarrow H^0(2K + (n+m)B). \label{27}
\end{eqnarray}

To prove this map surjects we restrict to a smooth curve $C \in |K+B|$ and use the fact that $X$ is regular. More precisely, we will use Lemma \ref{redcurves} and Proposition \ref{butler}.

Note that $H^1((n+m-2)B)  = 0$ and $K+B$ is base point free. 
So by Lemma \ref{redcurves},  we only need to prove that the following map is surjective:
$$H^0(K + (n+m-1)B|_C) \otimes H^0(K+B|_C) \rightarrow H^0(2K + (n+m)B|_C).$$

Let $F  = K + (n+m-1)B|_C$ and $E = K+B|_C$. We will now use Proposition \ref{butler}. 

We verify the following conditions:
\begin{eqnarray}
deg(F) > 2g(C); \label{24}\\
deg(F) + deg(E) > 4g(C)-2h^1(E)\label{25}
\end{eqnarray}

Note that $E$ is base point free and both $E$ and $F$ are semistable because they are line bundles. Further, $\mu(E) = deg(E)$. 

$deg(F) = (K+(n+m-1)B)\cdot (K+B)$ and $2g = (2K+B)\cdot (K+B) +2$.
Hence $deg(F) -2g = (K+(n+m-1)B - 2K -B) \cdot (K+B) -2$.

Thus \eqref{24} is equivalent to $((n+m-2)B - K) \cdot (K+B) > 2$. Or equivalently, 
$$(n+m-2)B^2 -K^2 + (n+m-3)B\cdot K > 2.$$
Since $B^2 \geq \frac{a}{b} B\cdot K$, we have (by Lemma \ref{1.2}) that,
$B \cdot K \geq \frac{a}{b}K^2$ and $B^2 \geq \frac{a^2}{b^2}K^2$. 

So $(n+m-2)B^2 -K^2 + (n+m-3)B\cdot K \geq (\frac{(n+m-2)a^2}{b^2}+\frac{(n+m-3)a}{b}-1)K^2$. 
Since $K^2 \geq 2$, we require $\frac{(n+m-2)a^2}{b^2}+\frac{(n+m-3)a}{b} > 2$. This follows from hypothesis (\ref{22}).

Now we check (\ref{25}).

\noindent
Claim: $h^1(E) = p_g$.

Consider the exact sequence 
$$ 0 \rightarrow \mathcal{O}(-C) \rightarrow \mathcal{O} \rightarrow \mathcal{O}|_{C} \rightarrow 0.$$ Tensoring this with $K+B$ we obtain
$$ 0 \rightarrow \mathcal{O} \rightarrow K+B \rightarrow  E \rightarrow 0.$$

This induces the long exact sequence 
$$ ..\rightarrow H^1(K+B) \rightarrow H^1(E) \rightarrow H^2(\mathcal{O}) \rightarrow H^2(K+B)\rightarrow..$$

$H^1(K+B) = H^2(K+B) =0$ by K-V vanishing and hence $h^2(\mathcal{O}) = p_g$. So the claim follows.

By the hypothesis on $p_g$, we conclude that $h^1(E) \geq 3$. 

$deg(F)+deg(E) = (K+(n+m-1)B+K+B)\cdot(K+B) = (2K+(n+m)B)\cdot(K+B).$

$4g(C) = 2(2g(C)-2)+4 = 2(2K+B)\cdot (K+B)+4 = \big{[}(4K+2B)\cdot(K+B)\big{]}+4$.

Thus (\ref{25}) is equivalent to 
$$(2K+(n+m)B)\cdot(K+B) > (4K+2B)\cdot(K+B)+4-2h^1(E)$$
$$\Leftrightarrow (2K+(n+m)B-4K-2B)\cdot(K+B) > 4-2p_g$$
$$\Leftrightarrow ((n+m-2)B-2K)\cdot(K+B) > -2, \text{~since}~p_g \geq 3$$
$$\Leftrightarrow (n+m-2)B^2+(n+m-4)B\cdot K-2K^2 > -2$$
$$\Leftarrow (\frac{(n+m-2)a^2}{b^2}+\frac{(n+m-4)a}{b}-2)K^2 > -2.$$
The last inequality follows from  the hypothesis (\ref{22}).

Thus the map (\ref{27}) is surjective and the theorem is proved. 
\end{proof}

Now we prove two statements about the $N_0$ property of $K+rB$. 
\begin{corollary}\label{n0_3}
Let $X$ be a minimal smooth regular surface of general type and let
$B$ be an ample base point free line bundle on $X$ such that $B^2 \geq \frac{a}{b} (B \cdot K)$ for some positive integers $a < b$.  Let $L = K + rB$ with $r \geq 3$. Then $L$ satisfies $N_0$ property if 
\begin{eqnarray}
\frac{(2r-2)a^2}{b^2} + \frac{(2r-4)a}{b} \geq 2, \label{277}\\
H^1\big{(}(2r-2)B\big{)} = 0. \label{28}
\end{eqnarray}
\end{corollary}
\begin{proof}
This is immediate from Theorem \ref{1}, by setting $n=m=r$.
\end{proof}

Now we state the result in the case  $B^2 \geq \frac{1}{n} (B \cdot K)$ for some $n \geq 2$.

\begin{corollary}\label{n0_4}
Let $X$ be a minimal smooth regular surface of general type and let
$B$ be an ample base point free line bundle on $X$ 
such that $B^2 \geq \frac{1}{n} (B \cdot K)$ for some $n \geq 2$. Let $L = K +rB$. Suppose that $H^1\big{(}(2r-2)B\big{)} = 0$. Then $L$ satisfies the $N_0$ property for $r \geq n+1$.
\end{corollary}
\begin{proof}
This follows easily from Corollary \ref{n0_3} by setting $a= 1, b=n$.

$r \geq n+1 \Rightarrow 2r-2 \geq 2n$ and $2r-4 \geq 2n-2 \Rightarrow \frac{2r-2}{n^2} + \frac{2r-4}{n} \geq \frac{2n}{n^2} + \frac{2n-2}{n} \geq 2$. 
\end{proof}

The above result improves Corollary \ref{n0_2}.

\begin{remark}\label{RedCurvesneeded}
{\rm In the proof of Theorem \ref{1}, we used reduction to curves to prove the surjectivity of 
\eqref{27}. As shown later in Example \ref{5.7}, reduction to curves is necessary to prove surjectivity as the CM lemma does not apply. For instance, let 
$n=m=3$ and $a=1, b=2$. Then \eqref{22} holds. We construct a surface $X$ of general type and a base point free and ample line bundle $B$ on $X$ such that
$B^2 \ge \frac{1}{2}(B \cdot K_X)$ and  $H^1(4B)=0$. So all the hypotheses of Theorem \ref{1} hold. However, $H^2(3B-K_X) = H^0(2K_X-3B) \ne 0$. Hence the surjectivity of \eqref{27} does \emph{not} follow from CM lemma.  
Indeed, we show that there are infinitely many examples of regular surfaces and line bundles on them, for which reduction to curves as in the proof of Theorem \ref{1} is necessary and gives better results. 
}
\end{remark}

Next, we determine when the line bundle $K+rB$ has $N_1$ property. 

\begin{theorem}\label{n1_2}
Let $X$ be a minimal smooth regular surface of general type and 
let $B$ be an ample base point free line bundle on $X$ such that $B^2 \geq \frac{a}{b} (B \cdot K)$ for some positive integers $a < b$. 
Let $L= K + rB$ for some $r \geq 3$. Suppose that $H^1((2r-3)B)  =0$. Then $L$ satisfies $N_1$ property if $r > \frac{b}{a}+\frac{3}{2}$. \end{theorem}
\begin{proof}
It is easy to see that the conditions (\ref{277}) and (\ref{28}) hold. So $L$ satisfies the $N_0$ property by Corollary \ref{n0_3}. 

To prove that $L$ satisfies the $N_1$ property, we prove that the following map is surjective:
\begin{eqnarray}
H^0(M_L \te  K \te B^{\te r}) \otimes H^0(K \te B^{\te r}) \rightarrow H^0(M_L \te 
K^{\te 2} \te B^{\te 2r}). \label{32}
\end{eqnarray}

First, we prove the surjectivity of:
$$H^0(M_L \te K \te B^{\te r}) \otimes H^0(B) \rightarrow H^0(M_L \te K \te B^{\te r+1}).$$

By CM lemma, this is true if the following two vanishings hold:
\begin{eqnarray}
H^1(M_L \te K \te B^{\te r-1} ) = 0,  \label{33} \\
H^2(M_L \te K \te B^{\te r-2} ) = 0. \label{34}
\end{eqnarray}

(\ref{33}) holds if the following map is surjective (since $H^1(K \te B^{\te r-1}) = 0)$:
$$H^0(K \te B^{\te r}) \otimes H^0(K \te B^{\te r-1}) \rightarrow H^0(K^{\te 2} \te B^{\te 2r-1}).$$

By Theorem \ref{1} (setting $n = r, m = r-1$), we need (\ref{22}) and (\ref{22a}) to hold. It is easy to see that they follow from hypotheses. 

(\ref{34}) follows because $H^1(L \te K \te B^{\te r-2} ) = 0$ and 
$H^2(K \te B^{\te r-2}) = 0$ (by K-V vanishing, since $r\geq 3$).

Repeating this process, we can absorb $(r-1)$ copies of $B$ to prove the surjectivity of the following map:
\begin{eqnarray*}
H^0(M_L \te K \te B^{\te 2r-2}) \otimes H^0(B) \rightarrow H^0(M_L \te K \te B^{\te 2r-1}).
\end{eqnarray*}
To show that (\ref{32}) surjects, it remains only to show that the following map surjects:
\begin{eqnarray}
H^0(M_L \te K \te B^{\te 2r-1}) \otimes H^0(K\te B) \rightarrow H^0(M_L \te K^{\te 2} \te B^{\te 2r}).\label{37}
\end{eqnarray}
We prove 
\begin{eqnarray}
H^1(M_L \te B^{\te 2r-2}) = 0, \label{35}\\
H^2(M_L \te K^{-1} \te B^{\te 2r-3} ) = 0\label{36}.
\end{eqnarray}

(\ref{35}) follows easily by Lemma \ref{horace} and  CM lemma (using Remark \ref{00}). 

As $H^1(L \te K^{-1} \te B^{\te 2r-3}) = 0$, (\ref{36}) holds (by Remark \ref{000}) if 
$H^2(K^{-1} \te B^{\te 2r-3}) = 0$.

By Serre duality, $H^2((2r-3)B - K ) = H^0(2K - (2r-3)B)$. If this group is nonzero, then 
$2K - (2r-3)B$ is linearly equivalent to an effective divisor and since $B$ is ample, 
$B \cdot (2K - (2r-3)B) \geq 0$. This implies that $2B\cdot K \geq (2r-3)B^2$. 

$B^2 \geq \frac{a}{b} B \cdot K \Rightarrow (2r-3)B^2 \geq \frac{(2r-3)a}{b} B \cdot K
\Rightarrow 2B\cdot K \geq \frac{(2r-3)a}{b} B \cdot K$. 

Since $\frac{(2r-3)a}{b}  > 2$, by hypothesis, we have a contradiction. So the map (\ref{37}), and hence (\ref{32}), is surjective and the theorem is proved. 
\end{proof}
\begin{corollary}\label{n1_3}
Let $X$ be a minimal smooth regular surface of general type and 
let $B$ be an ample base point free line bundle on $X$ such that $B^2 \geq \frac{1}{n} (B \cdot K)$ for some $n \geq 2$. Let $L = K +rB$. Suppose that $H^1((2r-3)B) = 0$ .
Then $L$ satisfies $N_1$ property for $r \geq n+2$.
\end{corollary}
\begin{proof}
This follows easily from Theorem \ref{n1_2} by setting $a = 1, b = n$.
\end{proof}

The above result improves the bound obtained in Corollary \ref{n1_1}.

It is interesting to determine, given a specific $r$, under which conditions does $K+rB$ have $N_1$ property. We do this in the following: 

\begin{corollary}\label{n1_5}
Let $B$ be an ample base point free line bundle on $X$. Let $r \geq 3$. Then $K+rB$ satisfies  $N_1$ property if $H^1((r-1)B) = 0$ and $B^2 \geq \frac{2}{2r-3} (B \cdot K)$.  
\end{corollary}
\begin{proof}
This follows easily from Theorem \ref{n1_2}: 

Setting $a = 2, b = 2r-3$, note that $r - 1 = \frac{b}{a}+\frac{1}{2} > \frac{b}{a}$. 
Hence the condition $H^1((r-1)B) = 0$ implies, by Lemma \ref{1.1}, $H^1(lB) = 0$ for all $l \geq r-1$.

Further, the hypotheses required in Theorem \ref{n1_2}   clearly follow in this case ($a = 2,~ b = 2r-3)$.
\end{proof}

By this corollary, we see that:

$K + 3B$ has $N_1$ property if $H^1(2B) = 0$ and $B^2 \geq \frac{2}{3}B \cdot K$;\hspace{.05in}
$K + 4B$ has $N_1$ property if $H^1(3B) = 0$ and $B^2 \geq \frac{2}{5}B \cdot K$; \hspace{.05in}
$K + 5B$ has $N_1$ property if $H^1(4B) = 0$ and $B^2 \geq \frac{2}{7}B \cdot K$; and so on.

\section{Higher Syzygies}\label{higher}
In this section $X$ represents a nonsingular minimal surface of general type and $B$ is a base point free ample divisor such that $K+B$ is base point free. 


Our goal is to specify $r$ for which $L = K + rB$ satisfies the $N_p$ property for $p \ge 2$. Let $n \ge2 $ be a positive integer. In this section we will assume either that $nB-K$ is nef or $(n+1)B-2K$ is nef. We get different bounds on $r$ under these two different hypotheses. 

By Theorem \ref{np}, to prove $N_p$ property for $L$, we have to prove that 
$H^1({M}^{\otimes a}_L \otimes L^{\otimes b}) = 0$ for all $1 \le a \le p+1$ and $b \ge 1$. 
To do so, we 
will use Lemmas \ref{horace} and \ref{key} repeatedly. In the process, we will need vanishings of the following form: 
\begin{eqnarray}\label{new}
H^1({M}^{\otimes a}_L \otimes  B^{\te m} \te K^{\te l}) = 0 \hspace{.15in}	\text{~~~~~~~~~~and~~~~~~~~~~~} 
\hspace{.15in}H^2({M}^{\otimes a}_L \otimes B^{\te m} \te K^{\te l}) = 0,
\end{eqnarray}
for various values of $a, m$ and $l$. Here $1 \le a \le p+1$ and $m \ge 1$, but $l$ is any integer, possibly negative. 

Most of this section is devoted to establishing \eqref{new}. We will prove a series of propositions to this end. Then using these results, 
we prove our main theorems. All our results will have two versions corresponding to the two hypotheses ($n \ge 2$ is a positive integer) :$$(n+1)B-2K {\rm~{is~~ nef~~  \hspace{.15in}or~}} \hspace{.15in}nB-K \rm{~~is~~ nef}.$$ 

We will start with a couple of easy lemmas. 

\begin{lemma}\label{van} Let $X$ be a smooth minimal surface of general type and let $B$ be a base point free ample divisor.  Let $n \ge 2$. Suppose that $(n+1)B-2K$ is nef. Then
$H^1(mB-lK) = H^2(mB-lK) = 0$, if $m > \frac{(n+1)(l+2)}{2}$.
\end{lemma}
\begin{proof} We write $mB-lK$ as $K$ plus a nef and big divisor and then use K-V vanishing.

Suppose that $l$ is even: say, $l = 2l'-2$. By the hypothesis, $m > (n+1)l'$. Write 

$mB-lK = 2K + l'((n+1)B-2K)+(m-(n+1)l')B$.

If $l$ is odd, we write $l = 2l'-1$ and $mB-lK = K + l'((n+1)B-2K)+(m-(n+1)l')B$.
\end{proof}

\begin{lemma}\label{van1} Let $X$ be a smooth minimal surface of general type and let $B$ be a base point free ample divisor.  Let $n \ge 2$. Suppose that $nB-K$ is nef. Then
$H^1(mB-lK) = H^2(mB-lK) = 0$, if $m > n(l+1)$.
\end{lemma}
\begin{proof} Note that  
$mB-lK = K + (l+1)(nB-K)+(m-n(l+1))B$. We are done by K-V vanishing.
\end{proof}

Now we will prove \eqref{new} for negative $l$.

\begin{proposition}\label{-K} Let $X$ be a minimal smooth surface of general type and let $B$ be a base point free ample divisor such that $K+B$ is base point free. Let $n \ge 2$ be a positive integer. Assume that $(n+1)B-2K$ is nef. Let $p \ge 1$, 
$r \ge n+p+1$ and $L = K + rB$. Then the following statements hold:
\begin{enumerate}
\item $H^1({M}^{\otimes p}_L \otimes B^{\te m} \te K^{-l}) = 0$ for $p \geq 1$, $l \geq 1$ and $m > \frac{(n+1)(l+2)}{2}+p+1$.\\
\item $H^2({M}^{\otimes p}_L \otimes  B^{\te m} \te K^{-l}) = 0$ for $p \geq 1$, $l \geq 1$ and $m > \frac{(n+1)(l+2)}{2}+p-1$.
\end{enumerate}
\end{proposition}
\begin{proof}
We first prove both statements of the theorem for $p =1$. 

To prove $H^1(M_L \otimes B^{\te m} \te K^{-l}) = 0$, it suffices to show that $H^1(B^{\te m} \te K^{-l}) = 0$ and that the following map is surjective:
$$H^0(B^{\te m} \te K^{-l}) \otimes H^0(L) \rightarrow H^0(L \otimes B^{\te m} \te K^{-l}).$$

$H^1(B^{\te m} \te K^{-l}) = 0$ by Lemma \ref{van}. The surjectivity of the above map can be proved by Lemma \ref{horace}. First $(r-1)$ copies of $B$ are captured and the surjectivity of these maps follows easily from CM lemma and Lemma \ref{van}. Note that we have $m > \frac{(n+1)(l+2)}{2}+2$. Finally, we have the map
$$H^0(B^{\te m+r-1} \te K^{-l}) \otimes H^0(K\te B) \rightarrow H^0(L\te B^{\te m} \te K^{-l}).$$ We show  that $H^1(B^{\te m+r-2} \te K^{-l-1})= H^2(B^{\te m+r-3} \te K^{-l-2}) = 0$: we only prove the later here and the former follows similarly. 

By hypothesis,  $r \ge n+2$   
and $m > \frac{(n+1)(l+2)}{2}+2 $. So 
$$m+r-3 >  \frac{(n+1)(l+2)}{2}+n+1 = \frac{(n+1)(l+4)}{2}.$$ 

So $H^2(B^{\te m+r-3} \te K^{-l-2}) = 0$ by Lemma \ref{van}. Hence (1) follows.

To prove (2) for $p=1$, we show that $H^1(L\te B^{\te m} \te K^{-l}) = 0$ and $H^2(B^{\te m} \te K^{-l}) = 0$. The first vanishing follows easily as above, and the second vanishing follows from Lemma \ref{van}. Note that 
$m > \frac{(n+1)(l+2)}{2}$.

Now let $p > 1$ and assume that the theorem is proved for $p-1$. 

We first prove (1) for $p$. Since
$H^1({M}^{\otimes p-1}_L \otimes B^{\te m} \te K^{-l}) = 0$ by induction hypothesis,  we only need to show that the following map is surjective:
$$H^0({M}^{\otimes p-1}_L \otimes B^{\te m} \te K^{-l}) \otimes H^0(L) \rightarrow H^0({M}^{\otimes p-1}_L \otimes B^{\te m} \te K^{-l} \te L).$$ As usual we use Lemma \ref{horace}. First we will capture one copy of $B$. 
$$H^0({M}^{\otimes p-1}_L \otimes B^{\te m} \te K^{-l}) \otimes H^0(B) \rightarrow H^0({M}^{\otimes p-1}_L \otimes B^{\te m+1} \te K^{-l}).$$ To prove this surjects we use the CM lemma: $H^1({M}^{\otimes p-1}_L \otimes B^{\te m-1} \te K^{-l}) = 
H^2({M}^{\otimes p-1}_L \otimes B^{\te m-2} \te K^{-l}) = 0$. These follow from induction hypothesis. Similarly, we capture $(r-1)$ copies of $B$.  

Then we have the following map:
$$H^0({M}^{\otimes p-1}_L \otimes B^{\te m+r-1} \te K^{-l}) \otimes H^0(K\te B) \rightarrow H^0({M}^{\otimes p-1}_L \otimes B^{\te m} \te K^{-l} \te L).$$ Using CM lemma, this map is surjective if 
\begin{eqnarray}
H^1({M}^{\otimes p-1}_L \otimes B^{\te m+r-2} \te K^{-l-1}) = 0,\label{4.1} \text{~and}\\
H^2({M}^{\otimes p-1}_L \otimes B^{\te m+r-3} \te K^{-l-2}) = 0.\label{4.2} 
\end{eqnarray}

Since $r \ge n+p+1$ and $m > \frac{(n+1)(l+2)}{2}+p+1$,    
$$m+r-2 > \frac{(n+1)(l+2)}{2}+2p+n \ge \frac{(n+1)(l+2)}{2}+n+2+p = \frac{(n+1)(l+4)}{2}+p+1.$$ So \eqref{4.1} holds  by induction hypothesis. 

Next we have $$m+r-3 > \frac{(n+2)(l+2)}{2}+2p+n-1 \ge \frac{(n+1)(l+2)}{2}+n+1+p
=\frac{(n+1)(l+4)}{2}+p.$$ Hence \eqref{4.2} holds by induction hypothesis. This completes the proof of (1) for $p$.

Finally, we prove (2) for $p$. We show that 
$H^1({M}^{\otimes p-1}_L \otimes L\te B^{\te m} \te K^{-l}) = 0$ and $H^2({M}^{\otimes p-1}_L \otimes B^{\te m} \te K^{-l}) = 0$. But both these vanishings follow easily from induction hypothesis. \end{proof}

Next we prove \eqref{new} for negative $l$ with the hypothesis $nB-K$ nef:
\begin{proposition}\label{-K1} Let $X$ be a minimal smooth surface of general type and let $B$ be a base point free ample divisor such that $K+B$ is base point free. Let $n \ge 2$. Assume that $nB-K$ is nef.
Let $p \ge 1, r \ge 2n+p+1$ and $L = K+rB$. Then the following statements hold:
\begin{enumerate}
\item $H^1({M}^{\otimes p}_L \otimes B^{\te m} \te K^{-l}) = 0$ for $p \geq 1$, $l \geq 1$ and $m > n(l+1)+p+1$.\\
\item $H^2({M}^{\otimes p}_L \otimes B^{\te m} \te K^{-l}) = 0$ for $p \geq 1$, $l \geq 1$ and $m > n(l+1)+p-1$.
\end{enumerate}
\end{proposition}
\begin{proof}
The proof is similar to the proof of Proposition \ref{-K}. We proceed by induction on $p$, with Lemma \ref{van1} playing the role of Lemma \ref{van}. 


To prove (1) for $p$, we proceed as in the proof of Proposition \ref{-K}. They key step 
is to show that the following map is surjective:
$$H^0({M}^{\otimes p-1}_L \otimes B^{\te m+r-1} \te K^{-l}) \otimes H^0(K\te B) \rightarrow H^0({M}^{\otimes p-1}_L \otimes B^{\te m} \te K^{-l} \te L).$$ Using CM lemma, this map is surjective if 
\begin{eqnarray}
H^1({M}^{\otimes p-1}_L \otimes B^{\te m+r-2} \te K^{-l-1}) = 0,\label{4.3} \text{~and}\\
H^2({M}^{\otimes p-1}_L \otimes B^{\te m+r-3} \te K^{-l-2}) = 0.\label{4.4} 
\end{eqnarray}

Since $r \ge 2n+p+1$ and $m > n(l+1)+p+1$,    
$$m+r-2 > n(l+3)+2p \ge n(l+2)+p.$$ So \eqref{4.3} holds  by induction hypothesis. 

Similarly $$m+r-3 > n(l+3)+2p-1 \ge n(l+3)+p-2.$$ So \eqref{4.4} holds  by induction hypothesis. 

The proof of (2) is the same as in the proof of Proposition \ref{-K}(2).
\end{proof}


\begin{remark}\label{3.5} {\rm 
Let $n \ge 1$. Assume that $(n+1)B-2K$ is nef or $nB-K$ nef. Since $B$ is base point free, $$((n+1)B-2K) \cdot B \ge 0 ~~~{~~\rm or~~}~~~~~~~~~ (nB-K) \cdot B \ge 0.$$ 
Hence, in either case, we get that $B^2 \ge \frac{1}{n}(B\cdot K)$.

Suppose further that $H^1((n+1)B) = 0$. Since $n+1 > n$, it follows, by Lemma \ref{1.1}, that 
$H^1(mB)=0$ for all $m \ge n+1$. }
\end{remark}

Next we establish \eqref{new} for positive $l$ with the additional hypothesis $H^1((n+1)B) = 0$.

First we will assume that $(n+1)B-2K$ is nef:

\begin{proposition}\label{main4}
Let $X$ be a minimal smooth surface of general type and let $B$ be a base point free ample divisor such that $K+B$ is base point free. Let $n \ge 2$. Assume that $(n+1)B-2K$ is nef and $H^1((n+1)B) = 0$.
Let $p \ge 1$, $r \ge n+p+1$ and $L = K+rB$.  Then the following statements hold:
\begin{enumerate}
\item $H^1({M}^{\otimes p}_L \otimes B^{\te m} \te K^{\te l}) = 0$ for $p \geq 1$, $l \geq 0$ and $m \ge n+p+1$.\\
\item $H^2({M}^{\otimes p}_L \otimes B^{\te m} \te K^{\te l}) = 0$ for $p \geq 1$, $l \geq 0$ and $m \ge  n+p$.
\end{enumerate}
\end{proposition}
\begin{proof}
We have $H^1(mB) = 0$ for all $m \geq n+1$, by Lemma \ref{1.1}.

We first prove both statements of the theorem for $p=1$. 

To prove (1), note that $H^1(B^{\te m} \te K^{\te l}) = 0$. Indeed, if $l > 0$ this follows from K - V vanishing. If $l=0$, use Lemma \ref{1.1}. Thus to prove (1) it suffices to prove that the following map is surjective:
$$H^0(B^{\te m} \te K^{\te l}) \otimes H^0(L) \rightarrow H^0(B^{\te m} \te K^{\te l}\te L).$$ If $l = 0$ or $l \geq 3$, this follows easily by CM lemma and K-V vanishing. If $l=1$, it follows from Theorem \ref{0}. If $l=2$, then we use CM lemma and most of the proof follows immediately from K-V vanishing. At the end we require the surjectivity of the following multiplication map:
$$H^0(B^{\te m+r-1} \te K^{\te 2}) \otimes H^0(K\te B) \rightarrow H^0(B^{\te m} \te K^{\te 2}\te L).$$ We show that $H^1(K\te B^{\te m+r-2}) = H^2(B^{\te m+r-3}) = 0$. 

The first follows by K-V vanishing. If the second does not hold, then by Serre duality we have $H^0(K - (m+r-3)B) \ne 0$. Then,  since $B$ is ample, we have $B\cdot K \ge (m+r-3)B^2$. But by hypothesis we have $(m+r-3)B^2 \ge \frac{(m+r-3)}{n} B \cdot K$, which implies that 
$B\cdot K \ge \frac{(m+r-3)}{n} B \cdot K$. But this is a contradiction because 
$m+r-3 > 2n+ 1>n$. 

To prove (2), we will show that $H^1(L\te B^{\te m} \te K^{\te l})=H^2(B^{\te m} \te K^{\te l})=0$.  The first vanishing is clear.  The second is clear if $l > 0$. If $l=0$, we argue as above using Serre duality and the fact that $m > n$.
 
Assume now that the theorem holds for $p-1$. 

We first prove (1) for $p$. Since $m \ge n+p+1$, we have 
$H^1({M}^{\otimes p-1}_L \otimes B^{\te m} \te K^{\te l}) = 0$, by induction hypothesis. So it suffices to show that the following map is surjective:
$$H^0({M}^{\otimes p-1}_L \otimes B^{\te m} \te K^{\te l}) \otimes H^0(L) \rightarrow H^0({M}^{\otimes p-1}_L \otimes B^{\te m} \te K^{\te l} \te  L).$$ We use Lemma \ref{horace}. We capture $(r-1)$ copies of $B$ using CM lemma and induction hypothesis. The necessary hypothesis on $m$ for (1) and (2) hold. We are then left with the following map:
$$H^0({M}^{\otimes p-1}_L \otimes B^{\te m+r-1} \te K^{\te l}) \otimes H^0(K\te B) \rightarrow H^0({M}^{\otimes p-1}_L \otimes B^{\te m} \te K^{\te l} \te L).$$ 

We use CM lemma:
\begin{eqnarray}
H^1({M}^{\otimes p-1}_L \otimes B^{\te m+r-2} \te K^{\te l-1}) = 0 \label{4.5} \text{~and}~\\
H^2({M}^{\otimes p-1}_L \otimes B^{\te m+r-3} \te K^{\te l-2}) = 0 \label{4.6}.
\end{eqnarray}

If $l \geq 1$, \eqref{4.5} follows from induction hypothesis. If $l = 0$, we use Proposition 
\ref{-K}. We need to show $m+r-2 > \frac{3(n+1)}{2}+p$. By hypothesis, $m, r \geq n+p+1$ and the required inequality is equivalent to $n/2+p > 3/2$. This is clear.  

\eqref{4.6} follows by induction if $l \ge 2$. If $l=0$, then it follows from Proposition \ref{-K} provided 
$m+r-3 > 2(n+1)+p-2$. This follows easily from the hypothesis on $m, r$. The case $l=1$ is similar.

Finally, to prove (2) for $p$,  we simply note that  
$H^1({M}^{\otimes p-1}_L \otimes L \te B^{\te m} \te K^{\te l}) = 
H^2({M}^{\otimes p-1}_L \otimes B^{\te m} \te K^{\te l}) =  0$, by induction hypothesis.
\end{proof}

We have a similar proposition with the hypothesis $nB-K$ nef:

\begin{proposition}\label{main41}
Let $X$ be a minimal smooth surface of general type and let $B$ be a base point free ample divisor such that $K+B$ is base point free. Let $n \ge 2$. Suppose that $nB-K$ is nef and $H^1((n+1)B) = 0$. 
Let $p \ge 1$, $r \ge 2n+p+1$ and $L = K+rB$. Then the following statements hold:
\begin{enumerate}
\item $H^1({M}^{\otimes p}_L \otimes B^{\te m} \te K^{\te l}) = 0$ for $p \geq 1$, $l \geq 0$ and $m \ge n+p+1$.\\
\item $H^2({M}^{\otimes p}_L \otimes B^{\te m} \te K^{\te l}) = 0$ for $p \geq 1$, $l \geq 0$ and $m \ge  n+p$.
\end{enumerate}
\end{proposition}
\begin{proof}
The proof is similar to the proof of Proposition \ref{main4}. We proceed by induction on $p$. 


To prove (1), the key step is to show that the following map is surjective:
$$H^0({M}^{\otimes p-1}_L \otimes B^{\te m+r-1} \te K^{\te l}) \otimes H^0(K\te B) \rightarrow H^0({M}^{\otimes p-1}_L \otimes B^{\te m} \te K^{\te l} \te L).$$ 

As before, we use CM lemma.  
\begin{eqnarray}
H^1({M}^{\otimes p-1}_L \otimes B^{\te m+r-2} \te K^{\te l-1}) = 0 \label{4.7} \text{~and}~\\
H^2({M}^{\otimes p-1}_L \otimes B^{\te m+r-3} \te K^{\te l-2}) = 0 \label{4.8}.
\end{eqnarray} 

To see \eqref{4.7}, note that it follows by induction if $l \ge 1$. If $l=0$, we use Proposition 
\ref{-K1}(1). The required condition is $m+r-2 > 2n+p$. This follows easily from the hypothesis on $r$ and $m$: $r \ge 2n+p+1$ and $m \ge n+p+1$.

\eqref{4.8} follows easily by induction hypothesis if $l \ge 2$. For $l = 0$, we use Proposition \ref{-K1}(2). By hypothesis on $m$ and $r$, we get 
$m+r-3 \ge 3n+2p-1 > 3n+2p-2$. To use Proposition \ref{-K1}, we need $m+r-3 > 3n+p-2$. The case $l=1$ is similar. 

The proof of (2) is easy and is similar to the proof of Proposition \ref{main4}(2).
\end{proof}

Now we will state our main results. These results will use the propositions proved above. 

\begin{theorem}\label{main5}
Let $X$ be a minimal smooth surface of general type and let $B$ be a base point free ample divisor such that $K+B$ is base point free. Let $n \ge 2$. Suppose that $H^1((n+1)B) = 0$ and $(n+1)B-2K$ is nef. Then $L = K+rB$ satisfies the $N_p$ property if $r \ge n+p+2$.
\end{theorem}
\begin{proof}
By Theorem \ref{np}, we know that $L$ satisfies $N_p$ property if 
$H^1(M_L^{\otimes a} \otimes L^{\otimes b}) = 0$ for all $1 \le a\leq p+1$ and $b \geq 1$. Since $r \geq n+p+2$, we have the required vanishing by Proposition 
\ref{main4}(1).
\end{proof}

Next we have a theorem with the hypothesis $nB-K$ nef: 
\begin{theorem}\label{main51}
Let $X$ be a minimal smooth surface of general type and let $B$ be a base point free ample divisor such that $K+B$ is base point free. Let $n \ge 2$. Suppose that $nB-K$ is nef and $H^1((n+1)B) = 0$. Then $L = K+rB$ satisfies the $N_p$ property if $r \ge 2n+p+1$.
\end{theorem}
\begin{proof}
$L$ satisfies $N_p$ property if 
$H^1(M_L^{\otimes a} \otimes L^{\otimes b}) = 0$ for all $1 \leq a \leq p+1$ and $b \geq 1$. 
By Remark \ref{n0n1}, we only need to verify this for $b=1$.

Since $r \geq 2n+p+1$,  Proposition 
\ref{main41}(1) shows that $H^1(M_L^{\otimes p} \otimes L) = 0$.
Thus $H^1(M_L^{\otimes p+1} \otimes L) = 0$ if the following map is surjective: 
$$H^0(M^{\otimes p}_L \otimes L) \otimes H^0(L) \rightarrow H^0(M^{\otimes p}_L \otimes L \otimes L).$$

We use Lemma \ref{horace} and CM lemma. First we have the map:
$$H^0(M^{\otimes p}_L \otimes L) \otimes H^0(B) \rightarrow H^0(M^{\otimes p}_L \otimes L \otimes B).$$

$r \geq 2n+p+1 \Rightarrow r-1 \ge 2n+p \ge n+p+1$. So 
$H^1(M^{\otimes p}_L \otimes K \otimes B^{\otimes r-1}) = 0$ by Proposition \ref{main41}(1). 

Since $r-2 \ge n+p-1$, $H^2(M^{\otimes p-1}_L \otimes K \otimes B^{\otimes r-2}) = 0$ by Proposition 
\ref{main41}(2). Since $H^1(M^{\otimes p-1}_L \otimes L \te K \otimes B^{\otimes r-2}) = 0$, 
it follows by Remark \ref{000}, that $H^2(M^{\otimes p}_L \otimes K \otimes B^{\otimes r-2}) = 0$.

Hence the above map is surjective and we can similarly capture $(r-1)$ copies of $B$. Then we have the map: 

$$H^0(M^{\otimes p}_L \otimes K \otimes B^{\otimes 2r-1}) \otimes H^0(K \otimes B) \rightarrow H^0(M^{\otimes p}_L \otimes L \otimes L).$$

Again we use CM lemma. First, $H^1(M^{\otimes p}_L \otimes B^{\otimes 2r-2}) = 0$, by 
Proposition \ref{main41}(1). We require $2r-2 \ge n+p+1$ and this is clear. Next,
$H^2(M^{\otimes p}_L \otimes K^{-1} \otimes B^{\otimes 2r-3}) =0$, by  
Proposition \ref{-K1}(2). The required inequality is $2r-3 > 2n+p-1$, which is also clear.
\end{proof}

\begin{remark} 
{\rm We can interpret these results in terms of regularity of the line bundle $B$: Let $X$ be a minimal smooth surface of general type and let $B$ be a base point free ample divisor such that $K+B$ is base point free. Suppose that the regularity of $B$ is $n+1$ for some integer $n \ge 2$. Let 
$L_r = K+rB$.

Then $L_r$ satisfies the $N_p$ property for $r \ge n+p+2$ if $(n+1)B-2K$ is nef and 
$L_r$ satisfies the $N_p$ property for $r \ge 2n+p+1$ if $nB-K$ is nef.
}
\end{remark}
\begin{remark} 
{\rm
In Proposition \ref{main4}, for statement (1) to hold (i.e., $H^1(M_L^{\otimes p} \otimes K \te B^{\te m}) = 0$) we require that \textit{both} $r$ and $m$ should be at least $n+p+1$. But to establish $N_p$ property, we need $H^1(M_L^{\otimes p+1} \otimes K \te B^{\te r}) = 0$.  So in our main result Theorem \ref{main5}, we needed to assume $r \ge n+p+1+1= n+p+2$.

On the other hand, in Proposition \ref{main41}, statement (1) holds (i.e., $H^1(M_L^{\otimes p} \otimes K \te B^{\te m}) = 0$) when $r \ge 2n+p+1$ and $m \ge n+p+1$. In other words, $m$ can be smaller than $r$. 
As a consequence of this observation, in our main result Theorem \ref{main51}, we only needed 
$r \ge 2n+p+1$.

We can improve Theorem \ref{main5} by only requiring $r \ge n+p+1$, if $X$ is regular. This is done below in Theorem \ref{main6}.
}
\end{remark}
\subsection{Regular Surfaces}\label{reg2}
In this subsection, we will assume that $X$ is a regular surface. We will improve Theorem \ref{main5}, using the reduction to curves to prove surjectivity of multiplications maps. This is analogous to the cases of $N_0$ and $N_1$ treated earlier in Section \ref{reg1}.  


First we have the following result improving Proposition \ref{main4}, by only requiring that 
$m \ge n+p$ for $H^1$ vanishing:
\begin{proposition}\label{5.1} Let $X$ be a minimal smooth regular surface of general type 
with $p_g \ge 3$ and $B$ be a base point free, ample divisor such that $B^2 \ge 5$. Let $n \ge 2$. Assume that $(n+1)B-2K$ is nef and $H^1((n+1)B)=0$.
Let $L = K+rB$.
Then  $H^1(M^{\otimes p}_L \otimes K \otimes B^{\otimes m}) = 0$, for $p \geq 1$, $r \ge n+p+1$ and $m \ge n+p$.
\end{proposition}
\begin{proof} 
Since $B^2 \ge 5$, by Remark \ref{B^2}, $K+B$ is base point free.
 
We proceed by induction on $p$. When $p = 1$, it suffices to prove that the following map is surjective:
$$H^0(K+rB) \otimes H^0(K+mB) \rightarrow H^0(2K+(r+m)B) .$$
We will use Theorem \ref{1}. The condition on $H^1$ vanishing follows from Lemma \ref{1.1}. The desired inequality is easy:
Since $r \geq n+2$ and $m \ge n+1$, we have
$\frac{r+m-2}{n^2}+\frac{r+m-4}{n} \geq  \frac{2n+1}{n^2}+\frac{2n-1}{n} = \frac{2n^2+n+1}{n^2} \geq 2$.
 
When $p =2$, we prove that the following map is surjective:
$$H^0(M_L \otimes K \te B^{\te m}) \otimes H^0(K \te B^{\te r}) \rightarrow H^0(M_L \otimes K^{\te 2} \te B^{\te m+r}).$$
We first consider the following map:
$$H^0(M_L \otimes K \te B^{\te m}) \otimes H^0(B) \rightarrow H^0(M_L \otimes K \te B^{\te m+1}).$$
To see this map surjects, we will apply CM lemma. $H^1$ vanishes because we already proved the proposition for $p=1$. To prove $H^2(M_L \otimes K \te B^{\te m-2}) = 0$, notice that 
$H^1(L \te K \te B^{\te m-2}) = 0$ and $H^2(K \te B^{\te m-2}) = 0$.

After absorbing $(r-1)$ copies of $B$, we have 
$$H^0(M_L \otimes K \te B^{\te m+r-1}) \otimes H^0(K \te B) \rightarrow H^0(M_L \otimes K^{\te 2} \te B^{\te m+r}). $$
To prove this map surjects, we will again use CM lemma. $H^1$ vanishing follows from 
Proposition  \ref{main4}. To prove $H^2(M_L \otimes K^{ -1} \te B^{\te m+r-3}) = 0$,  we use 
Proposition \ref{-K}(2).

Now let $p \geq 3$. Assume that the proposition is true for $p-1$. Let $r \ge n+p+1$ and $m \ge n+p$.  It is enough to prove that the following map is surjective:
$$H^0(M^{\otimes p-1}_L \otimes K \otimes B^{\otimes m}) \otimes H^0(K \otimes B^{\otimes r}) \rightarrow H^0(M^{\otimes p-1}_L \otimes K^{\otimes 2} \otimes B^{\otimes m+r}).$$ First step is to prove that the following map is surjective:
$$H^0(M^{\otimes p-1}_L \otimes K \otimes B^{\otimes m}) \otimes H^0(B) \rightarrow H^0(M^{\otimes p-1}_L \otimes K  \otimes B^{\otimes m+1}).$$

Let $C \in |B|$ be a smooth curve. $H^1(M^{\otimes p-1}_L \otimes K \otimes B^{\otimes m-1})=0$ by induction.  
By Proposition \ref{redcurves}, it is enough to show that the following map is surjective:
$$H^0(M^{\otimes p-1}_L \otimes K \otimes B^{\otimes m} \otimes \mathcal{O}_C) \otimes H^0(B\otimes \mathcal{O}_C) \rightarrow H^0(M^{\otimes p-1}_L \otimes K  \otimes B^{\otimes m+1} \otimes \mathcal{O}_C).$$ This map is surjective if the following map is surjective, by Proposition \ref{gp}:
$$H^0(M^{\otimes p-1}_{L\otimes \mathcal{O}_C} \otimes K \otimes B^{\otimes m} \otimes \mathcal{O}_C) \otimes H^0(B\otimes \mathcal{O}_C) \rightarrow H^0(M^{\otimes p-1}_{L\otimes \mathcal{O}_C} \otimes K  \otimes B^{\otimes m+1} \otimes \mathcal{O}_C).$$ Note that $H^1(L -B) = 0$.
Now we use Theorem \ref{butler}.  
Let $F = M^{\otimes p-1}_{L\otimes \mathcal{O}_C} \otimes K \otimes B^{\otimes m} \otimes \mathcal{O}_C$ and $E = B \otimes \mathcal{O}_C$.   $E$ is base point free. 
Let $g = genus(C)$. We have $2g-2 = (K+B) \cdot B$.
We show that 
\begin{enumerate}
\item $F$ is semistable;
\item $\mu(F) > 2g$;
\item $\mu(F) > 4g-deg(E)-2h^1(E)$.
\end{enumerate}

By Theorem 1.2 \cite{1}, $M_{L\otimes \mathcal{O}_C}$ is semistable if $deg(L \otimes \mathcal{O}_C) \ge 2g$. Since  
$deg(L \otimes \mathcal{O}_C) = (K+rB) \cdot B$ and  $2g = (K+B) \cdot B +2$, the required condition is equivalent to $(r-1)B^2 \ge 2$. This definitely holds because $r \ge n+p+1$.  
It is known that tensor product of semistable vector bundles on a curve is also semistable. So it follows that $F$ is semistable. 

Theorem 1.2 \cite{1} also gives: 
$\mu(M_{L\otimes \mathcal{O}_C}) \geq -2$. Hence $\mu(M^{\otimes p-1}_{L\otimes \mathcal{O}_C}) \geq -2(p-1)$ and
$$\mu(F)  \ge -2(p-1)+ (K+mB) \cdot B.$$ Thus (2) is equivalent to $(m-1)B^2 > 2+2(p-1) = 2p$. This holds because $B^2 \geq 5$, $n \ge 2$ and $m \ge n+p$. 

Since $deg(E) = B^2$, and $4g = [(2K+2B)\cdot B]+4$, (3) is equivalent to 
$$(K+mB)\cdot B - 2(p-1) > (2K+2B)\cdot B +4-B^2-2h^1(E).$$ This in turn is equivalent to 

$$(m-1)B^2 - B\cdot K > 2p+2-2h^1(E).$$ 

Since $nB^2 \geq B \cdot K$, 
$(m-1)B^2 - B\cdot K \ge  (m-1)B^2 - nB^2 = (m-n-1)B^2$.  As $B^2 \ge 5, (m-n-1)B^2 \ge 5(m-n-1)$. By hypothesis, $m \geq n+p$. So 
$5(m-n-1) \ge 5(p-1)$ and the required inequality follows if we show $5(p-1) > 2p+2 -2h^1(E)$. Since $h^1(E) \ge 0$, we have to prove $5(p-1) > 2p+2$.
This is clear because $p \ge 3$. 

After gathering $r-1$ copies of $B$, we have the following morphism:
$$H^0(M^{\otimes p-1}_L \otimes K \otimes B^{\otimes m+r-1}) \otimes H^0(B\te K) \rightarrow H^0(M^{\otimes p-1}_L \otimes K^{\te 2}  \otimes B^{\otimes m+r}).$$

We use CM lemma: $H^1(M^{\otimes p-1}_L  \otimes B^{\otimes m+r-2}) = 0$, by 
Proposition \ref{main4}. We require that $m+r-2 \ge n+p$, and this is clear because by 
hypothesis  $m \ge n+p$ and $r \ge n+p+1$.
Similarly
$H^2(M^{\otimes p-1}_L \otimes K^{-1} \otimes B^{\otimes m+r-3}) = 0$, by 
Proposition \ref{-K}. 
\end{proof}

Now we prove our main theorem: 

\begin{theorem}\label{main6}
Let $X$ be a minimal smooth regular surface of general type with $p_g \ge 3$ and let $B$ be a base point free, ample divisor such that $B^2 \ge 5$. Let $n \ge 2$. Suppose that $(n+1)B-2K$ is nef and $H^1((n+1)B)=0$.  Then $K+rB$ has the $N_p$ property for $r \ge n+p+1$.
\end{theorem}
\begin{proof} We proceed by induction on $p$. 
Corollaries \ref{n0_4} and \ref{n1_3} establish the theorem for $p = 0$ and 1 respectively. It is easy to see that the required hypotheses are true. Note that $B^2 \ge \frac{1}{n}(B\cdot K)$ by Remark \ref{3.5}.

Suppose that $p > 1$. Let $r \ge n+p+1$. 

$L = K+rB$ has $N_p$ property if $H^1({M}^{\otimes a}_L \otimes L^{\otimes b})  = 0$ for $1 \leq a \leq p+1$ and $b \ge 1$.  

Since $r \ge n+p+1$, $L$ has $N_{p-1}$ by induction. So
we only need to show that $H^1({M}^{\otimes p+1}_L \otimes L^{\otimes b})  = 0$ for all $b \ge 1$. We will show this only for $b=1$ as the other cases follow from Remark \ref{n0n1}.

Since $H^1(M^{\otimes p}_L \otimes L)  = 0$, it is enough to show that the following map is surjective:
$$H^0(M^{\otimes p}_L \otimes L) \otimes H^0(L) \rightarrow H^0(M^{\otimes p}_L \otimes L \otimes L).$$

The first step is to prove the following map is surjective. 
$$H^0(M^{\otimes p}_L \otimes L) \otimes H^0(B) \rightarrow H^0(M^{\otimes p}_L \otimes L \otimes B).$$
We use the CM lemma. $H^1(M^{\otimes p}_L \otimes K \otimes B^{\otimes r-1}) = 0$ by Proposition \ref{5.1}. 
To prove that 
$H^2(M^{\otimes p}_L \otimes K \otimes B^{\otimes r-2}) = 0$, note that $H^1(M^{\otimes p-1}_L \otimes L \otimes  K \otimes B^{\otimes r-2}) = 0$ by Proposition \ref{main4}(1) and $H^2(M^{\otimes p-1}_L \otimes  K \otimes B^{\otimes r-2}) = 0$ by Proposition \ref{main4}(2). Note that $r-2 \ge n+p-1$.

After  similarly absorbing $(r-1)$ copies of $B$ we are left with the following map:
$$H^0(M^{\otimes p}_L \otimes K \otimes B^{\otimes 2r-1}) \otimes H^0(K \otimes B) \rightarrow H^0(M^{\otimes p}_L \otimes L \otimes L).$$

We will prove this map is surjective using the CM lemma. $H^1(M^{\otimes p}_L \otimes B^{\otimes 2r-2}) = 0$ by Proposition \ref{main4}(1). We require 
$2r-2 \ge n+p+1$. This is clear because $r \ge n+p+1$.  $H^2(M^{\otimes p}_L \otimes K^{-1}  \otimes B^{\otimes 2r-3}) = 0 $, by Proposition \ref{-K}(2). We require 
$2r-3 > \frac{3(n+1)}{2} + p-1$. This is also clear. 
\end{proof}

\section{Examples}  \label{examples}

\begin{example}\label{5.1}{\rm  Let $\phi: S \rightarrow \mathbb{P}^2$ be  a double cover branched along a smooth curve of degree $10$ in $\mathbb{P}^2$. Then $S$ is a smooth minimal regular surface of general type. 

We have $K_S = {\phi}^{\star}(\mathcal{O}_{\mathbb{P}^2}(2))$. Set $B = {\phi}^{\star}(\mathcal{O}_{\mathbb{P}^2}(1))$. Then $B$ is base point free, ample and $H^1(B) = 0$. $B^2 = 2$ and $B \cdot K = 4$. 

It can be checked that $K_S + 2B$ does not satisfy $N_0$ property, but $K_S + 3B$ does. This illustrates Corollary \ref{n0_4}. }
\end{example}

\begin{example}\label{5.2}{\rm
Let $\phi: S \rightarrow \mathbb{P}^2$ be  a triple cover branched along a smooth curve of degree $9$ in $\mathbb{P}^2$. Again $S$ is a smooth minimal regular surface of general type. 

We have $K_S = {\phi}^{\star}(\mathcal{O}_{\mathbb{P}^2}(3))$. Set $B = {\phi}^{\star}(\mathcal{O}_{\mathbb{P}^2}(1))$. Then $B$ is base point free, ample and $H^1(lB) = 0$ for all $l \ge 1$. Also, $B^2 = 3$ and $B \cdot K = 9$. 

Now $K_S + 2B$ does not satisfy $N_1$ property. But according to Corollary \ref{n1_3}, $K_S+5B$ does. }
\end{example}

The above two examples are studied in Section 5 of \cite{2}. Refer to it for more details. 

\begin{example}\label{5.3} This example shows that there are infinitely many instances when the hypotheses of Theorem \ref{main6} are satisfied. Moreover, the hypothesis that 
$(n+1)B-2K$ is nef will hold in a strict sense in these examples: namely, 
$(n+1)B-2K$ is nef but $nB-2K$ is not nef. 


{\rm
Let $n \ge 2$. We construct a surface $X$ such that the following conditions hold: 
\begin{itemize}
\item $X$ is a minimal smooth regular surface of general type with $p_g \ge 3$.
\item $B$ is a base point free and ample line bundle on $X$ such that 
$B^2 \ge 5$ and $H^1((n+1)B)=0$.
\item $(n+1)B-2K$ is nef and $nB-2K$ is not nef. 
\end{itemize}

Let $S$ denote the Hirzebruch surface $\mathbb{F}_1$. Namely,  $\pi: S \rightarrow \mathbb{P}^1$ is the projective space bundle associated to the locally free sheaf of rank two $\mathcal{O}_{\mathbb{P}^1} \oplus \mathcal{O}_{\mathbb{P}^1} (-1)$. 
Let $C_0$ denote a section of $S$ that represents the tautological line bundle 
$\mathcal{O}_S(1)$ and let $f$ denote a fiber. 
Then $K_S = -2C_0-3f$. 

We have 
${\pi}_{\star}(\mathcal{O}_S) = \mathcal{O}_{\mathbb{P}^1}$, 
${\pi}_{\star}(\mathcal{O}_S(f)) = \mathcal{O}_{\mathbb{P}^1}(1)$.
Further,  for $a \ge 1$, 
$${\pi}_{\star}(\mathcal{O}_S(aC_0)) =  {Sym}^a\big{(}\mathcal{O}_{\mathbb{P}^1}\oplus \mathcal{O}_{\mathbb{P}^1}(-1)\big{)} = 
\mathcal{O}_{\mathbb{P}^1}\oplus \mathcal{O}_{\mathbb{P}^1}(-1) \oplus \mathcal{O}_{\mathbb{P}^1}(-2) \oplus \ldots \oplus \mathcal{O}_{\mathbb{P}^1}(-a).$$
Now, for some for  $m \geq 1$, let
$\phi: X \rightarrow S$ be a double cover branched along a smooth curve in 
the very ample linear system 
$|6C_0+2(m+3)f|$ on $S$. 

Then $ {\phi}_{\star}(\mathcal{O}_X) = \mathcal{O}_S \oplus \mathcal{O}_S(-3C_0-(m+3)f)$.  
We have that $K = K_X = {\phi}^{\star}(C_0+mf)$. Let $B = {\phi}^{\star}(C_0+bf)$ for some $b > 1$.  

Our goal is to choose $m, b$ appropriately so that the above conditions hold. 

Note that $C_0^2 = -1, C_0 \cdot f = 1$ and $f^2 = 0$. So $B^2 = 2(-1+2b)$ and $B \cdot K = 2(-1+b+m)$.

Then $X$ is a minimal smooth surface of general type.  
We see easily that $B$ is ample and base point free, $B^2 \geq 5$ and $K+B$ is base point free. 

Claim: Let $a \ge 0$. Then $H^1(\mathcal{O}_S(aC_0+bf))=0$ if $b > a-2$.

Proof: By the projection formula, $H^1(\mathcal{O}_S(aC_0+bf)) = H^1\big{(}\mathcal{O}_{\mathbb{P}^1}(b)\oplus \mathcal{O}_{\mathbb{P}^1}(b-1) \oplus \ldots 
\oplus \mathcal{O}_{\mathbb{P}^1}(b-a)\big{)} $. This is zero if $b-a > -2$. So the claim follows.

\underline{$X$ is regular}: $H^1(\mathcal{O}_X) = H^1(\mathcal{O}_S) \oplus 
H^1\big{(}\mathcal{O}_S(-3C_0-(m+3)f)\big{)}$. 
By Serre duality, 
$H^1\big{(}\mathcal{O}_S(-3C_0-(m+3)f)\big{)} = H^1\big{(}\mathcal{O}_S(C_0+mf)\big{)}.$
This is zero by the claim. Hence 
$H^1(\mathcal{O}_X) = H^1(\mathcal{O}_{\mathbb{P}^1})  = 0$. 

\underline{$p_g(X) \ge 3$}: 
${\phi}_{\star}(K_X) ={\phi}_{\star}{\phi}^{\star}(C_0+mf) = \mathcal{O}_S(C_0+mf) \oplus \mathcal{O}_S(-2C_0-3f)$ 
and ${\pi}_{\star}(\mathcal{O}_S(C_0+mf)) = \mathcal{O}_{\mathbb{P}^1}(m)\oplus  \mathcal{O}_{\mathbb{P}^1}(m-1)$ . 
As $m \ge1$, 
$h^0(\mathcal{O}_{\mathbb{P}^1}(m)) \ge 3$. 
Since $h^0(X,K_X) = h^0(S, {\phi}_{\star}{\phi}^{\star}(C_0+mf)) \ge h^0(\mathcal{O}_{\mathbb{P}^1}(m))$, we conclude that $p_g \ge 3$. 

We now verify that $H^1(X, (n+1)B)=0$. 

$H^1(X, (n+1)B) = H^1\big{(}\mathcal{O}_S((n+1)C_0+((n+1)b)f)\big{)}\oplus H^1\big{(}\mathcal{O}_S((n-2)C_0+((n+1)b-m-3)f\big{)}$.



By the lemma, $H^1\big{(}\mathcal{O}_S((n+1)C_0+((n+1)b)f)\big{)}=0$ if $(n+1)b > n+1-2$. This is obvious. 

$H^1\big{(}\mathcal{O}_S((n-2)C_0+((n+1)b-m-3)f\big{)}=0$ if $(n+1)b-m-3> n-4$. 

Thus $H^1((n+1)B)=0$ if $m < (n+1)b-n+1$.

Let $D = s C_0 + t f$ be a divisor on $S$, with $s > 0$. If $D$ is nef, then $D \cdot C_0 \geq 0$, which gives $t \ge s$. On the other hand, if $t > s$, then $D$ is ample, hence nef. 
Further, $D$ is nef if and only if ${\phi}^{\star}(D)$ is nef. This follows from projection formula, and the observation that, for a curve $C$ in $S$, 
$D \cdot C =  2{\phi}^{\star}(D) \cdot {\phi}^{\star}(C)$.

$$(n+1)B-2K = {\phi}^{\star}\big{(}(n-1)C_0+((n+1)b-2m)f\big{)}.$$ This is nef if $(n+1)b-2m > n-1$. 
Similarly 
$$nB-2K = {\phi}^{\star}\big{(}(n-2)C_0+(nb-2m)f\big{)}.$$ This is not nef if $nb-2m < n-2$.

Thus we have the following three inequalities:
\begin{enumerate}
\item $m <  (n+1)b-n+1$;
\item $(n+1)b-2m > n-1 \Leftrightarrow m < (\frac{n+1}{2})b+\frac{1-n}{2}$;
\item $nb-2m < n-2 \Leftrightarrow m > \frac{n}{2}b+\frac{2-n}{2}$.
\end{enumerate} 

Thinking of $b$ as the horizontal and $m$ as vertical axis, these three inequalities represent the appropriate half planes determined by three lines in the $(b,m)$-plane. Denoting these lines $l_1, l_2$ and $l_3$, respectively, we observe that their slopes are given by 
$m_1= n+1$, $m_2 = \frac{n+1}{2}$ and $m_3 = \frac{n}{2}$ and we have 
$m_3 \le m_2 \le m_1$.

The solutions to the inequality (1) lie below $l_1$, solutions to inequality (2) lie below line $l_2$ and  
solutions to the inequality (3) lie above $l_3$. Thanks to the inequalities among the slopes, we have infinitely many simultaneous solutions to the three
inequalities. }
\end{example}

\begin{example}\label{5.4}  This example shows that there are infinitely many instances when the hypotheses of Theorem \ref{main51} are satisfied. Moreover, the hypothesis that 
$nB-K$ is nef will hold in a strict sense in these examples: namely, 
$nB-K$ is nef but $(n-1)B-K$ is not nef. 

{\rm
Let $n \ge 2$. We construct a surface $X$ such that the following conditions hold: 
\begin{itemize}
\item $X$ is a minimal smooth surface of general type.
\item $B$ is a base point free and ample line bundle on $X$ such that 
$K+B$ is base point free and $H^1((n+1)B)=0$.
\item $nB-K$ is nef and $(n-1)B-K$ is not nef. 
\end{itemize}

Let $S =  \mathbb{F}_1$ as in the above example. For some $m \ge 1$, 
let 
$\phi: X \rightarrow S$ be a double cover branched along a smooth curve in 
the very ample linear system 
$|6C_0+2(m+3)f|$ on $S$.  

Then 
$K = K_X = {\phi}^{\star}(C_0+mf)$. Let $B = {\phi}^{\star}(C_0+bf)$ for some $b >1$. Our goal is to choose $m, b$ appropriately so that above conditions hold.  

We see that $X$ is a minimal smooth surface of general type, $B$ is a base point, ample divisor, and $K+B$ is base point free. 

As in Example \ref{5.3},  $H^1((n+1)B)=0$ if $m <  (n+1)b-n+1$.

$nB-K = {\phi}^{\star}((n-1)C_0+(nb-m) f)$ is nef if $nb-m > n-1$ and

$(n-1)B-K = {\phi}^{\star}((n-2)C_0+((n-1)b-m) f)$ is not nef if $(n-1)b-m  < n-2$. 

Thus the required inequalities are:
\begin{enumerate}
\item $m < (n+1)b-n+1$;
\item $nb-m > n-1  \Leftrightarrow m <  nb-n+1$;
\item $(n-1)b-m < n-2 \Leftrightarrow m > (n-1)b-n+2$.
\end{enumerate} 

We may argue as in Example 5.3, noting that since $n+1 > n > n-1$, the slope of line determining (2) is more than the slope of line determining (3). Alternatively, we observe that $(n-1)b-n+2 < nb-n+1 < (n+1)b-n+1$, if $b > 1$. So we can choose a suitable $m$ for any large value of $b$.}
\end{example}

\begin{example}\label{5.5}{\rm We can construct infinitely many examples as above for arbitrary finite covers $X \rightarrow S=\mathbb{F}_1$ of any degree.  }
\end{example}
 
\begin{example}\label{5.7} We construct infinitely many examples such that the hypotheses of Theorem \ref{1} hold and for which general arguments using CM lemma will not yield the results. This is because the vanishing of $H^2$ required in CM lemma does \textit{not} hold. Further, for these examples, we have $B^2 < B\cdot K$. This shows that our conclusions do not follow from results in \cite{2}. 


{\rm 
Let $n,m$ be positive integers such that $n+m$ is even and  $n+m \ge 6$. Note that we are often only interested in the case $n=m \ge 3$. So the conditions on $n,m$ are not serious. We will show that there exist positive integers $a,b$ and a surface $X$ such that the following conditions are satisfied:
\begin{itemize}
\item $X$ is a minimal smooth regular surface of general type with $p_g \ge 3$ and $K^2 \ge 2$.
\item $B$ is a base point free and ample line bundle on $X$ such that 
$B\cdot K > B^2 \ge \frac{a}{b}(B \cdot K)$ and
$$\frac{(n+m-2)a^2}{b^2} + \frac{(n+m-4)a}{b} \geq 2.$$
\item $H^1\big{(}(n+m-2)B\big{)} = 0$.
\item $H^2((n+m-3)B-K_S) \ne 0$.
\end{itemize}



Set $a = 2$ and $b = n+m-2$. By the assumptions on $n,m$, we get that $b \ge 4$. 

Then $$\frac{(n+m-2)a^2}{b^2} + \frac{(n+m-4)a}{b} = \frac{4b}{b^2} + \frac{2(b-2)}{b} = 2.$$ 

So one of the above conditions holds for $a$ and $b$.
We will now construct a surface $X$ for which the other conditions are satisfied. 




Let $S$ denote the Hirzebruch surface $\mathbb{F}_1$. 

Given positive integers $1 \le r < s$, let 
$\phi: X \rightarrow S$ be a double cover branched along a smooth curve $C$ in 
the very ample linear system $|2(r+2)C_0+2(s+3)f|$ on $S$.

Then $K_X = {\phi}^{\star}(rC_0+sf)$. 
Let 
$B = {\phi}^{\star}(C_0+2f)$.

Clearly $X$ is a minimal smooth surface of general type, $B$ is a base point, ample divisor, $K+B$ is base point free. Further, $K^2 \ge 2$, $B^2 = 6$ and $B\cdot K = 2(r+s)$.

For simplicity, set $i = r+2$ and $j = s+3$. 

\underline{$X$ is regular}: $H^1(\mathcal{O}_X) = H^1(\mathcal{O}_S) \oplus 
H^1\big{(}\mathcal{O}_S(-iC_0-jf)\big{)}$.
By Serre duality, 
$H^1\big{(}\mathcal{O}_S(-iC_0-jf)\big{)} = H^1\big{(}\mathcal{O}_S((i-2)C_0+(j-3)f)\big{)}$.
By the claim in Example \ref{5.3}, $H^1\big{(}\mathcal{O}_S((i-2)C_0+(j-3)f)\big{)} = 0$ if 
$j-3 > i-4$. Equivalently, $j-i > -1$, which is clear.  
Hence 
$H^1(\mathcal{O}_X) = H^1(\mathcal{O}_S) = H^1(\mathcal{O}_{\mathbb{P}^1}) = 0$. 

\underline{$p_g \ge 3$}: 
${\phi}_{\star}(K_X) = \mathcal{O}_S(rC_0+sf) \oplus \mathcal{O}_S(-2C_0-3f)$ 
and 
${\pi}_{\star}(\mathcal{O}_S(rC_0+sf)) = \mathcal{O}_{\mathbb{P}^1}(s)\oplus  \mathcal{O}_{\mathbb{P}^1}(s-1) \oplus \ldots \oplus \mathcal{O}_{\mathbb{P}^1}(s-r) $ . 
As $s \ge1$, 
$h^0(\mathcal{O}_{\mathbb{P}^1}(s)) \ge 3$. 
Since $h^0(X,K_X) \ge h^0(S, \mathcal{O}_S(rC_0+sf)) \ge h^0(\mathcal{O}_{\mathbb{P}^1}(s))$, we conclude that $p_g \ge 3$. 

Next we have the following:  
\begin{eqnarray}\label{cond1}
B\cdot K > B^2 \ge \frac{2}{b} (B\cdot K) \Leftrightarrow r+s > 3 \ge \frac{2}{b}(r+s).
\end{eqnarray}
On the other hand, $2K-(b-1)B = {\phi}^{\star}((2r-b+1)C_0+(2s-2b+2)f)$. 

So $H^0(2K-(b-1)B) \ne 0$ if 
\begin{eqnarray}\label{cond2}
2r-b+1 > 0 ~~~~~~~~{\rm~~~ and~~~}~~~~~~~~~ 2s-2b+2>2r-b+1.
\end{eqnarray}
Set $s = b$ and $r = b/2$. Note that $b$ is even.  Then clearly \eqref{cond1} and \eqref{cond2} hold. 

Next, we check that $H^1(bB)=0$. 

We have:
$H^1(X,bB) = H^1(S, bC_0+2bf) \oplus H^1(S, (b-i)C_0+(2b-j)f)).$ By the claim in Example \ref{5.3}, these groups are zero if $2b > b-2$ and $2b-j > b-i-2$. The first inequality is 
obvious. For the second inequality, recall that $i=r+2$ and $j=s+3$. So the required inequality is $b > s-r-1$. This is clear because $s=b$ and $r=b/2$.

Thus we have shown that given any $n,m$ such that $n+m$ is even and at least $6$, there exists a surface $X$ of general type and an ample, base point free divisor $B$ on $X$ and positive integers $a<b$ such that the hypotheses of Theorem \ref{1} are satisfied and $H^2((n+m-3)B-K_X) \ne 0$. 


}
\end{example}
\begin{remark}\label{5.6}{\rm For any surface $X$ and any ample divisor $B$ on $X$, it is well known that given any $p$, $K_X+rB$ has $N_p$ property for sufficiently large $r$. For instance, this follows from the main results of \cite{5} (for any surface) and \cite{2} (for surfaces of general type). Our Theorem \ref{main51} gives another proof of this fact for surfaces of general type. 

Indeed, if $B$ is ample, some multiple of $B$ is very ample, and in particular, base point free. Further $K+mB$ is base point free for large $m$. So replacing $B$ by a large enough power, we can suppose that $B$ and $K+B$ are base point free. Moreover, again since $B$ is ample, we have $H^1(mB) = 0$ and $-K + mB$ is ample (hence nef) for large $m$. So the hypothesis of Theorem \ref{main51} are satisfied for a suitable multiple of $B$. Then we conclude that $K+ rB$ has $N_p$ property for sufficiently large $r$.}
\end{remark}

\bibliographystyle{plain}

\end{document}